\documentclass[11pt,a4paper]{article}
\usepackage{amssymb,amsmath}
\usepackage[multiple]{footmisc}

\usepackage{amssymb}

\newtheorem{theorem}{Theorem}[section]
\newtheorem{corollary}[theorem]{Corollary}
\newtheorem{definition}[theorem]{Definition}
\newtheorem{lemma}[theorem]{Lemma}
\newtheorem{proposition}[theorem]{Proposition}
\newtheorem{remark}[theorem]{Remark}

\newtheorem{example}[theorem]{Example}

\newenvironment{proof}{\begin{trivlist}\item[]{\it Proof.}}
{\hfill$\square$\end{trivlist}}

\newcommand{\specificthanks}[1]{\@fnsymbol{#1}}

\def\cocoa{{\hbox{\rm C\kern-.13em o\kern-.07em C\kern-.13em o\kern-.15em A}}}

\def\mz{{\mathbb{Z}}}
\def\mn{{\mathbb{N}}}
\def\mq{{\mathbb{Q}}}
\def\mr{{\mathbb{R}}}

\def\mb{{\mathbb{B}}}
\def \bx{{\pmb x}}
\def\bn{{\pmb n}}
\def\bu{{\pmb u}}
\def\bv{{\pmb v}}
\def\b0{{\pmb 0}}
\def\ba{{\pmb a}}
\def\bE{{\pmb E}}
\def\bd{{\pmb d}}
\def\bw{{\pmb w}}

\def\rma{{\mathbb{T}}}
\def\zma{\mathbb{Z}_{\max}}

\def\genps{generalized powers }
\def\gp{GP}
\def\QC{{QC }}
\def\diag{{\Delta}}
\def\gann{{\overline{Ann}}}

\def\<{\langle}
\def\>{\rangle}

\makeatletter
\let\@fnsymbol\@alph
\makeatother
\begin{document}

\title{Prime congruences of additively idempotent semirings and a Nullstellensatz for tropical polynomials}
\author{D\'aniel Jo\'o\thanks{This research was partially supported by 
National Research, Development and Innovation Office,  NKFIH K 119934, NKFIH PD 121410 and the exchange project "Combinatorial ring theory" betwen the Bulgarian and Hungarian Academies of Sciences.}\ \textsuperscript{, }\thanks{R\'enyi Institute of Mathematics, Hungarian Academy of Sciences, Budapest, Hungary, Email: joo.daniel@renyi.mta.hu }
 \ and Kalina Mincheva
 \thanks{Yale University, New Haven, CT 06511, USA, Email: kalina.mincheva@yale.edu}
 }

\date{}
\maketitle 
{\small \begin{center} 
\end{center}
}


\begin{abstract} 
A new definition of prime congruences in additively idempotent semirings is given using twisted products. This class turns out to exhibit some analogous properties to the prime ideals of commutative rings. In order to establish a good notion of radical congruences it is shown that the intersection of all primes of a semiring can be characterized by certain twisted power formulas. A complete description of prime congruences is given in the polynomial and Laurent polynomial semirings over the tropical semifield $\rma$, the semifield $\zma$ and the two element semifield $\mb$. The minimal primes of these semirings correspond to monomial orderings, and their intersection is the congruence that identifies polynomials that have the same Newton polytope. It is then shown that the radical of every finitely generated congruence in each of these cases is an intersection of prime congruences with quotients of Krull dimension $1$. An improvement of a result from \cite{BE13} is proven which can be regarded as a Nullstellensatz for tropical polynomials.
\end{abstract}




\section{Introduction}\label{sec:intro}

Tropical geometry, that is geometry over the tropical semiring $\rma =\mathbb{R}_{\max} = \{ \mathbb{R} \cup \{-\infty\}, \max, + \}$, is an area that recently has received a lot of interest and attention and has applications not just to algebraic geometry, but also to intersection theory, mirror symmetry and mathematical biology. 

Two of the semifields that we take into account in this paper - $\mathbb{B}$ and $\zma$ - are of interest to arithmetic geometry. They are key to the development of the semiring approach to characteristic one geometry taken up in \cite{Les12}, \cite{CC13} and \cite{CC14}. The semifield $\zma$ is central to the theory in \cite{CC13} which aims at finding a correct framework for characteristic one geometry that is in congruence with the original idea of J. Tits \cite{Tit56}.

Classically, a tropical variety (as defined in \cite{MS} and \cite{Mik06}) is a balanced polyhedral complex. 
However, recently there has been a lot of work aiming at finding the appropriate definition of a tropical scheme. The authors in \cite{GG13} and \cite{MR14} endow varieties defined over an idempotent semiring with tropical scheme structure. The set of points of this variety - called the bend loci - is defined by polynomial equations coming from a certain congruence relation. A different approach was taken in \cite{IR14}, where so-called supertropical structures were studied in order to establish the Zariski correspondence between congruences of tropical polynomials and algebraic sets.

In the case of idempotent semirings congruences are a more natural object to consider than ideals. Even though the fundamental objects of classical algebraic geometry are the prime ideals of commutative rings, ideals of semirings do not fulfill the same role as they are no longer in bijection with the congruences of the base structure. 

It is a natural approach then, which was taken up in \cite{BE13}, to try to transfer the notion of primeness to congruences in a way that the resulting structures exhibit nice properties and analogies with classical algebraic geometry. A possibility, which was investigated in \cite{Les12}, is to require that in the quotient by a prime congruence there are no zero divisors. The main drawback of this approach is that the prime property of a congruence solely depends on the equivalence class of the $0$ element (i.e. the kernel of the congruence), which in general contains little information about the congruence itself. For example, in a Laurent polynomial semiring over a semifield the kernel of every congruence is just $\{0\}$ (see Proposition \ref{prop:kernels}). A stricter way to define primes, as in \cite{BE13} and \cite{Lor12} is to require that their quotients are cancellative semirings, i.e. $ab = ac$ implies $a=0$ or $b=c$.  While this certainly is a narrower class, congruences with this property fail to be intersection indecomposable in general, making it difficult to treat them analogously to the primes of ring theory. Moreover, most structures that are of interest to us will contain infinitely long chains of congruences with cancellative quotients (see Corollary \ref{cor: infinitechains}), hence they do not provide a good notion of Krull dimension.

\par\smallskip
In our approach, so called twisted products of pairs of elements are used to define prime congruences. The twisted product of two ordered pairs $(a,b)$ and $(c,d)$ is the ordered pair $(ac+bd,ad+bc)$. The key heuristic is provided by the fact that in ring theory an ideal $P$ is prime if and only if for any elements $a \not\equiv b\;\textrm{mod}\;P$ and $c \not\equiv d\;\textrm{mod}\;P$ we have $ad+bc \not\equiv ac+bd\;\textrm{mod}\;P$. Following this characterization we define primes to be the congruences that do not contain twisted product of pairs that lie outside the congruence. To relate this notion to the above mentioned studies we show in Theorem \ref{thm: ind QC are prime} that congruences that are prime in our sense are precisely the intersection indecomposables with cancellative quotients.  
A first natural objective in studying prime congruences is to describe the set $Rad(I)$, defined as the intersection of all primes that contain the congruence $I$. In order to do this we introduced certain twisted power formulas called {\it generalized powers} for ordered pairs, and showed in Theorem \ref{thm:radical} that the elements of a pair are congruent in $Rad(I)$ precisely when some generalized power of that pair lies in $I$. \par\smallskip
Our next goal was to understand the prime congruences of the polynomial and Laurent polynomial semirings over the semifields $\mb$, $\zma$ and $\rma$. In all of these cases minimal primes turn out to correspond to monomial orderings. Applying a result of Robbiano from \cite{Rob85} that classifies monomial orderings, it can be then shown that every prime congruence of these semirings can be described by a certain {\it defining  matrix}, whose number of rows will equal the dimension of the quotient by that prime. As a consequence the dimension of a $k$-variable polynomial or Laurent polynomial semiring is $k$ over $\mb$ and $k+1$ over $\rma$ or $\zma$. This result meets our intuitive expectations, since the semifield $\mb$ is of dimension $0$ and the semifields $\zma$ and $\rma$ are of dimension $1$. Furthermore, using this description of prime congruences we show that two polynomials with coefficients in $\mb$ are congruent in every prime if and only if their Newton polytopes are the same. Consequently, the quotient of the polynomial algebra over $\mb$ by the intersection of all prime congruences (i.e. the radical of the trivial congruence) can be described as the semiring of lattice polytopes with the sum of two polytopes being the convex hull of their union and the product the Minkowski sum. Similar descriptions can be given in all of the other studied cases.
\par\smallskip
We note that the points of the tropical affine space $\rma^n$ can be identified with prime congruences of $\rma[\bx]$, whose quotient algebra is $\rma$. We call these {\it geometric congruences} and study them in Section \ref{sec:Nullst}, where we aim at understanding solutions of finite sets of tropical polynomial equations. 
With the above identification tropical varieties can be thought of as the set of geometric congruences containing a fixed bend congruence (in the sense of \cite{GG13} and \cite{MR14}).
\par\smallskip
A key component of the Nullstellensatz of classical algebraic geometry is that in a polynomial ring over a field every radical ideal is the intersection of maximal ideals (i.e. it is a Jacobson ring). One can not expect this to hold for congruences of polynomial semirings, since there are very few maximal congruences. However, one obtains an analogous result if the maximal congruences are replaced with prime congruences with $1$ dimensional quotient. In Theorem \ref{thm: 1dim} (i) it is shown that for any finitely generated congruence $E$ in a polynomial or Laurent polynomial semiring over $\mb$, $\zma$ or $\rma$, $Rad(E)$ is the intersection of the primes that contain $E$ and have a quotient with dimension at most $1$. \par\smallskip

Finally, we apply the methods developed to prove Theorem \ref{thm: 1dim} (i) to improve a Nullstellensatz type result from \cite{BE13}. In their approach one thinks of the elements of the k-variable semiring $\rma[\bx]$ as functions on the set $\rma^k$, and for a congruence $E$ denotes by $V(E)$ the subset of $\rma^k$ where every congruent pair from $E$ gives the same value. On the other hand for a subset $H$ of $\rma^k$ they denote by $\bE(H)$ the congruence that identifies polynomials that agree on every point of $H$. In this terminology the aim of a "tropical Nullstellensatz" is to describe the set $\bE(V(E))$ for a finitely generated congruence $E$. To achieve this in \cite{BE13} a set denoted by $E_+$ is defined using formulas that are similar to our generalized powers and it is shown that $E \subseteq E_+ \subseteq \bE(V(E))$ and $V(E) = V(E_+)$; moreover, a certain limit construction is given to describe the set $\bE(V(E))$ in general. We improved this result in Theorem \ref{thm: 1dim} (ii) by showing that in fact for any finitely generated congruence we have $E_+ = \bE(V(E))$ and consequently the set $E_+$ is always a congruence. We regard this result as our version of a Nullstellensatz for tropical polynomials. \par\smallskip

This paper is organized as follows. In Section \ref{sec:primes} we define the main objects that we work with throughout the paper, including the notion of prime property for idempotent semirings. We conclude the section by showing that a congruence is prime if and only if it is intersection indecomposable and its quotient is cancellative. Section \ref{sec:radicals} contains our results regarding the radical of congruences, in particular its description using generalized powers. In Section \ref{sec:polysemirings} we give a full description of the primes of the polynomial and Laurent polynomial semirings over $\mb$, $\zma$ and $\rma$, along with some related results such as calculating the dimension in each case. Section \ref{sec:Nullst} contains our results regarding finitely generated congruences and the improvement of the "tropical Nullstellensatz" from \cite{BE13}.

\par
\subsection*{Acknowledgements} We would like to thank Diane Maclagan, Felipe Rinc\'on and Jeffrey Giansiracusa for their insightful comments and interesting discussions. \par
\par\bigskip

\section{Prime congruences of semirings}\label{sec:primes}

In this paper by a {\it semiring} we mean a commutative semiring with multiplicative unit, that is a nonempty set $R$ with two binary operations $(+,\cdot)$ satisfying:

\begin{itemize}
\item[(i)] $(R,+)$ is a commutative monoid with identity element $0$
\item[(ii)] $(R,\cdot)$ is a commutative monoid with identity element $1$
\item[(iii)] For any $a,b,c \in R$: $a(b+c) = ab+ac$
\item[(iv)] $1 \neq 0$ and $a\cdot 0 = 0$ for all $a \in R$
\end{itemize}

A {\it semifield} is a semiring in which all nonzero elements have multiplicative inverse. We will denote by $\mb$ the semifield with two elements $\{1,0\}$, where $1$ is the multiplicative identity, $0$ is the additive identity and $1+1 = 1$. The {\it tropical semifield}  $\rma$ - sometimes also denoted by $\mr_{\max}$ - is defined on the set $\{-\infty\} \cup \mr$, by setting the $+$ operation to be the usual maximum and the $\cdot$ operation to be the usual addition, with $-\infty$ playing the role of the $0$ element. In this paper we will use the exponential notation $t^c ,\; c\in \mr$ for the elements of $\rma$, allowing us to write $1 = t^0$ for the multiplicative identity element and $0$ for the additive identity element. The semifield $\zma$ is just the subsemifield of integers in $\rma$. \par\smallskip
A polynomial (resp. Laurent polynomial) ring with variables $\bx = (x_1,\dots,x_k)$ over a semifield $F$ is the semiring, denoted by $F[\bx]$ (resp. $F(\bx)$),  whose elements are formal linear combinations of the monomials $\{x_1^{n_1}...x_k^{n_k}\mid\;n_i \in \mn\}$ (resp. $\{x_1^{n_1}...x_k^{n_k}\mid\;n_i \in \mz\}$) with coefficients in $F$, with addition and multiplication being defined in the usual way. For an integer vector $\bn = (n_1,\dots,n_k)$ we will use the notation $\bx^\bn = x_1^{n_1}...x_k^{n_k}$.\par\smallskip
As usual, an {\it ideal} in the semiring $R$ is just a subsemiring that is closed under multiplication by any element of $R$. Congruences of semirings are just operation preserving equivalence relations.

\begin{definition}{\rm
A {\it congruence} $I$ of the  semiring $R$ is a subset of $R \times R$ satisfying
\begin{itemize}
\item[(C1)] For $a \in R$, $(a,a) \in I$
\item[(C2)] $(a,b) \in I$ if and only if $(b,a) \in I$
\item[(C3)] If $(a,b) \in I$ and $(b,c) \in I$ then $(a,c) \in I$
\item[(C4)] If $(a,b) \in I$ and $(c,d) \in I$ then $(a+c,b+d) \in I$
\item [(C5)] If $(a,b) \in I$ and $(c,d) \in I$ then $(ac,bd) \in I$
\end{itemize}
}
\end{definition}

The unique smallest congruence is the diagonal of $R \times R$ which is denoted by $\diag$, also called the {\it trivial congruence}. $R \times R$ itself is the {\it improper congruence} the rest of the congruences are called {\it proper}. Quotients by congruences can be considered in the usual sense, the quotient semiring of $R$ by the congruence $I$ is denoted by $R/I$. The {\it kernel} of a congruence is just the equivalence class of the $0$ element. Note that kernels do not determine the congruences, for instance non-trivial congruences can have $\{0\}$ as their kernel. The kernel of a congruence is always an ideal, and when we say that the kernel of a congruence is generated by some elements, we will mean it is generated as an ideal by those elements. We will say that the kernel of a congruence is trivial if it equals $\{0\}$.\par\smallskip 
As usual, if $\varphi:R_1\rightarrow R_2$ is a morphism of semirings, and $I$ is a congruence of $R_2$, the preimage of $I$ is the congruence $\varphi^{-1}(I)=\{(\alpha_1,\alpha_2)\in R_1\times R_1\mid (\varphi(a_1),\varphi(a_2))\in I\}$. By the {\it kernel of a morphism} $\varphi$ we mean the preimage of the trivial congruence $\varphi^{-1}(\diag)$, it will be denoted by $Ker(\varphi)$. If $R_1$ is a subsemiring of $R_2$ then the restriction of a congruence $I$ of $R_2$ to $R_1$ is $I|_{R_1}=I\cap R_1 \times R_1$.\par\smallskip
By a $\mb$-algebra we simply mean a commutative semiring with idempotent addition (that is $a+a = a, \forall a$).  Throughout this section $A$ denotes an arbitrary $\mb$-algebra.
Note that the idempotent addition defines an ordering via \[a \geq b \iff a+b = a.\]
Elements of $A \times A$ are called {\it pairs}. We denote pairs by Greek letters, and denote the coordinates of the pair $\alpha$ by $\alpha_1,\alpha_2$. The {\it twisted product} of the pairs $\alpha = (\alpha_1, \alpha_2)$ and $\beta = (\beta_1, \beta_2)$ is $(\alpha_1\beta_1+\alpha_2\beta_2,\alpha_1\beta_2+\alpha_2\beta_1)$. Note that the twisted product is associative and the pairs form a monoid under this operation, with the pair $(1,0)$ being the identity element. For the rest of the paper in any formula containing pairs the product is always the twisted product, so the twisted product of $\alpha$ and $\beta$ is simply denoted by $\alpha\beta$ . Similarly $\alpha^n$ denotes the twisted $n$-th power of the pair $\alpha$, and we use the convention $\alpha^0 = (1,0)$. The product of two congruences $I$ and $J$ is defined as the congruence generated by the set $\{\alpha\beta\mid \alpha \in I \:\beta\in J\}$. For an element $a$ and a pair $\alpha$ we define their product as $a(\alpha_1,\alpha_2)=(a\alpha_1,a\alpha_2)$ which is the same as the twisted product $(a,0)\alpha$.
\par\smallskip
The following elementary properties of congruences play an important role,
\begin{proposition}\label{prop: congbasic}
Let $I$ be a congruence of $A$,
\begin{itemize}
\item[(i)] For $\alpha \in I$ and an arbitrary pair $\beta$ we have $\alpha\beta \in I$.
\item[(ii)] For any two congruences $I$ and $J$ we have $IJ \subseteq I\cap J$.
\item[(iii)] If $(a,b) \in I$ and $a \leq c \leq b$ then $(a,c) \in I$ and $(b,c) \in I$. In particular if $(a,0) \in I$ then for every $a\geq c$ we have $(c,0) \in I$.
\end{itemize}
\end{proposition}
\begin{proof}
(i) follows immediately from the definition of a congruence and (ii) follows from (i). For (iii) consider that in $A/I$ we have that \[a = b\Rightarrow c=a+c=b+c=b=a.\]
\end{proof}

One can readily show that for usual commutative rings, an ideal is prime if and only if the corresponding congruence does not contain twisted products of pairs lying outside. This motivates the following definition.

\begin{definition}{\rm We call a congruence $P$ of a  $\mb$-algebra $A$  prime if it is proper and for every $\alpha, \beta \in A \times A$ such that $\alpha\beta \in P$ either $\alpha \in P$ or $\beta \in P$. We call a $\mb$-algebra a {\it domain} if its trivial congruence is prime.
}
\end{definition}

We define dimension similarly to the Krull-dimension in ring theory:
\begin{definition}{\rm
By {\it dimension} of a $\mb$-algebra $A$ we will mean the length of the longest chain of prime congruences in $A \times A$ (where by length we mean the number of strict inclusions). The dimension of $A$ will be denoted by $dim(A)$.
}
\end{definition}

For the above definition to make sense one needs to verify that every $\mb$-algebra $A$ has at least one prime congruence. Indeed it is a known fact that $\mb$ is the only simple $\mb$-algebra (i.e. the only proper congruence is the trivial one). Hence by the usual Zorn's lemma argument we see that every $\mb$-algebra has a proper congruence with quotient $\mb$, and it follows from the definition that such a congruence is prime. For the sake of completeness we provide a short proof of the above fact:

\begin{proposition}
The only simple $\mb$-algebra is $\mb$.
\end{proposition}
\begin{proof}
First assume that $A$ is a $\mb$-algebra without zero-divisors. Then the map $\varphi: A\rightarrow \mb$ defined as $\varphi(x) = 1$ for $x \neq 0$ and $\varphi(0) = 0$ is a homomorphism of $\mb$-algebras. Hence $Ker(\varphi)$ is a proper congruence of $A$, which can only be trivial when $A \simeq \mb$.

Now assume that there are - not necessarily distinct - non-zero elements $x,y \in A$ such that $xy = 0$. Then it is easy to verify that $$C = \{(a,b)\in A\times A\mid\smallskip xa = xb\}$$ is a congruence of $A$, which is non-trivial since $(y,0) \in C$ and proper since $(1,0)\notin C$. Hence $A$ is not simple.
\end{proof}

A congruence is called {\it intersection indecomposable}  if it can not be obtained as the intersection of two strictly larger congruences. 

\begin{proposition} \label{prop:primes are indec}
If a congruence is prime then it is intersection indecomposable. 
\end{proposition}
\begin{proof}
Indeed if $P$ is the intersection of the strictly larger congruences $I$ and $J$, then take $\alpha \in I \setminus P$ and $\beta \in J \setminus P$. Now by part (i) of Proposition \ref{prop: congbasic} we have that $\alpha\beta \in I \cap J = P$ so $P$ can not be prime.
\end{proof}

\begin{remark}\label{rem:zariski-topology}{\rm
As a consequence of Proposition \ref{prop:primes are indec} one can define the Zariski topology on the set of prime congruences in the usual way, by setting the closed sets to be the ones that contain a fixed congruence. }
\end{remark}

A $\mb$-algebra $A$ is called {\it cancellative} if whenever $ab = ac$ for some $a,b,c \in A$ then either $a = 0$ or $b=c$. The {\it annihilator} of a pair $\alpha$ is defined as $Ann_A(\alpha) = \{\beta\in A\times A \mid\:\alpha\beta \in \diag\}$. $Ann_A(\alpha)$ satisfies the axioms (C1)-(C2) and (C4)-(C5) of a congruence but in general it is not transitive, consider the following example:
\begin{example}{\rm
Let $A$ be the algebra $\mb[x,y]/\<(y,y^2)\>$. Then it is easy to check that $(y,x+1),(y,1) \in Ann_A((x,x+y))$ but $(1,x+1) \notin Ann_A((x,x+y))$.
}
\end{example}

The annihilator of an element $a \in A$ is defined as the annihilator of the pair $(a,0)$ and is also denoted by $Ann_A(a)$. It is easy to verify the following properties:

\begin{proposition}\label{prop:primeprop}
\begin{itemize}
\item[(i)]For any $a \in A$, $Ann_A(a) = \{\beta\in A\times A \mid\: a\beta_1 = a\beta_2 \}$, moreover $Ann_A(a)$ is a congruence.
\item[(ii)]A is cancellative if and only if for every element $a\neq 0$ we have $Ann_A(a) = \diag$, and A is a domain if and only if for every pair $\alpha \notin \diag$ we have $Ann_A(\alpha) = \diag$.
\item[(iii)] For a congruence $I$ the quotient $A/I$ is cancellative if and only if for every element $a$ and pair $\alpha$ such that $(a,0)\alpha \in I$ either $(a,0) \in I$ or $\alpha \in I$.
\item[(iv)]If P is a prime congruence, then $A/P$ is cancellative.
\item[(v)] If P is a prime congruence of $A_1$, $\varphi:A_2\rightarrow A_1$ is a morphism of $\mb$-algebras and $A_3$ is a subalgebra of $A_1$, then $\varphi^{-1}(P)$ and $P|_{A_3}$ are prime congruences.
\end{itemize}
\end{proposition}

We will call a $\mb$-algebra {\it totally ordered} if its addition induces a total ordering. The next proposition shows that $\mb$-algebras which are domains are always totally ordered.

\begin{proposition}\label{prop:primeorder}
\begin{itemize}
\item[(i)] A $\mb$-algebra which is a domain is totally ordered. 
\item[(ii)] If a $\mb$-algebra $A$ is totally ordered then the trivial congruence of $A$ is prime if and only if $A$ is cancellative.
\end{itemize}
\end{proposition}
\begin{proof}
For (i) let $A$ be a domain and $x,y\in A$ two arbitrary elements. We have that $$(x+y,x)(x+y,y) = (x^2+y^2+xy, x^2+y^2+xy) \in \diag.$$ Since the trivial congruence is prime either $(x+y,x) \in \diag$ or $(x+y,y) \in \diag$, so indeed at least one of $x \geq y$ or $y \geq x$ hold. For (ii) one direction is clear by (iv) of Proposition \ref{prop:primeprop}. For the other direction assume that $A$ is a totally ordered and cancellative. Let $\alpha, \beta$ be two pairs satisfying $\alpha\beta \in \diag$. We can assume that $\alpha_1 \geq \alpha_2$, $\beta_1 \geq \beta_2$ and $\alpha_1\beta_2 \geq \alpha_2\beta_1$. Now we have that $$\alpha\beta = (\alpha_1\beta_1 +\alpha_2\beta_2, \alpha_1\beta_2 + \alpha_2\beta_1) = (\alpha_1\beta_1, \alpha_1\beta_2) \in \diag.$$ 
Then since $A$ is cancellative either $\beta \in \diag$ or $(\alpha_1,0) \in \diag$ which, by $\alpha_1 \geq \alpha_2$ implies $\alpha_1 = \alpha_2 = 0$ so $\alpha \in \diag$.
\end{proof}

A congruence $I$ for which $A/I$ is cancellative will be called {\it quotient cancellative} or {\it QC} for short. The main result of this section shows that QC congruences are prime if and only if they are intersection indecomposable. 

\begin{lemma}\label{cancellative powers}
Let $A$ be a cancellative $\mb$-algebra, and $\alpha \in A\times A$ a pair. If for some integer $n >0$ we have $\alpha^n \in \diag$ then $\alpha \in \diag$.
\end{lemma}
\begin{proof}
First let us assume $\alpha^2 \in \diag$. It follows that $\alpha_1^2 + \alpha_2^2 = \alpha_1\alpha_2$, and then $$\alpha_1^2\alpha_2 = \alpha_1^3 + \alpha_1\alpha_2^2 \geq \alpha_1\alpha_2^2$$
and similarly $\alpha_1\alpha_2^2 \geq \alpha_1^2\alpha_2$ so we have that $\alpha_1^2\alpha_2 = \alpha_1\alpha_2^2$.  Now by cancellativity either  $\alpha_1$ or $\alpha_2$ is $0$ but then since $\alpha^2 \in \diag$ both are $0$, or neither is $0$ and then after dividing by $\alpha_1\alpha_2$ we obtain $\alpha_1 = \alpha_2$. Now in the general case if $\alpha^n \in \diag$ then every power of $\alpha$ greater than $n$ is in $\diag$, in particular for some $k$ we have $\alpha^{2^k} \in \diag$ and we are done by applying the first half of the argument.
\end{proof}

\begin{lemma}\label{lem:cancellative annulator}
Let $A$ be a cancellative $\mb$-algebra, then for any pair $\alpha \in A\times A$ the set $Ann_A(\alpha)$ is a congruence.
\end{lemma}
\begin{proof}
If $\alpha \in \diag$ then $Ann_A(\alpha)=A\times A$, which is a congruence. Assume now that $\alpha \notin \diag$. The axioms (C1),(C2),(C4) and (C5) are easy to verify. For transitivity consider some pairs $(x,y)$ and $(y,z)$ for which we have $(x,y)\alpha \in \diag$ and $(y,z)\alpha \in \diag$. Since $\alpha \notin \diag$ and $A$ is cancellative we can assume that none of $x,y,z$ is $0$. We will show that  \[\beta:= (y+z,0)(x,z) \alpha=((y+z)x,(y+z)z)\alpha \in \Delta\] and since $y+z$ non zero this will imply $(x,z)\alpha \in \Delta$.
Expanding the above we obtain:
$$(\beta_1,\beta_2) = ((y+z)x,(y+z)z)(\alpha_1,\alpha_2)=(yx\alpha_1 + yz\alpha_2 + zx\alpha_1+z^2\alpha_2, yx\alpha_2 + yz\alpha_1 + zx\alpha_2 + z^2\alpha_1)$$
By symmetry it suffices to show that $\beta_1 \geq \beta_2$ (with respect to the ordering that comes from the idempotent addition). 
We have that $\beta_1 \geq z(y\alpha_2+x\alpha_1)$ and since $(x,y)\alpha \in \diag$ we obtain
$$\beta_1 = yx\alpha_1 + yz\alpha_2 + zx\alpha_1+z^2\alpha_2 + zx\alpha_2 + zy\alpha_1$$
Now we have $z(z\alpha_2+y\alpha_1)$ amongst the terms, using $(y,z)\alpha \in \diag$ we get:
$$\beta_1 = yx\alpha_1 + yz\alpha_2 + zx\alpha_1+z^2\alpha_2 + zx\alpha_2 + zy\alpha_1 + z^2\alpha_1+zy\alpha_2$$
We obtained $\beta_1 \geq x(y\alpha_1+z\alpha_2)$, using $(y,z)\alpha \in \diag$ again we get: 
$$\beta_1 = yx\alpha_1 + yz\alpha_2 + zx\alpha_1+z^2\alpha_2 + zx\alpha_2 + zy\alpha_1 + z^2\alpha_1+zy\alpha_2 + xz\alpha_1 + xy\alpha_2$$
and finally from $\beta_1 \geq z(x\alpha_1+y\beta_2)$ and $(x,y)\alpha \in \diag$ we obtain:
$$\beta_1 = yx\alpha_1 + yz\alpha_2 + zx\alpha_1+z^2\alpha_2 + zx\alpha_2 + zy\alpha_1 + z^2\alpha_1+zy\alpha_2 + xz\alpha_1 + xy\alpha_2 + zy\alpha_1+zx\alpha_2 $$
which is indeed bigger than $\beta_2$, which is the sum of the 5th, 7th, 10th and 11th terms. 
Hence $Ann_A(\alpha)$ is a congruence. 
\end{proof}

\begin{theorem}\label{thm: ind QC are prime}Let $A$ be a $\mb$-algebra. A congruence $I$ is prime if and only if it is \QC and intersection indecomposable.
\end{theorem}
\begin{proof}
It follows from Proposition \ref{prop:primes are indec} and Proposition \ref{prop:primeprop} that prime congruences are \QC  and intersection indecomposable. For the other direction, taking the quotient by $I$, we can assume that $I = \diag$ is QC and intersection indecomposable (so $A$ itself is cancellative). Note that this can be done because all three properties depend on the quotient of the congruence. If $\diag$ is not prime there exists an element $\alpha \notin \diag$ such that $Ann_A(\alpha) \neq \diag$. By the previous lemma $Ann_A(\alpha)$ is a congruence. Let $Q = \bigcap_{\beta \in Ann_A(\alpha)}Ann_A(\beta)$. Q is a congruence (as an intersection of congruences), and since $\alpha \in Q$ we have $\diag \subsetneq Q$. Clearly $Ann_A(\alpha)Q = \diag$, we claim that $ Ann_A(\alpha) \cap Q = \diag$. Otherwise suppose that $\beta \in (Ann_A(\alpha) \cap Q)\setminus \diag$, since  $Ann_A(\alpha)Q = \diag$ we have that $\beta^2 \in \diag$, and then by Lemma \ref{cancellative powers} we have $\beta \in \diag$ completing the proof.
\end{proof}


\section{Radicals of congruences}\label{sec:radicals}
Our next objective is to establish the notion of radicals of congruences and provide a similar algebraic description to the one in ring theory. 

\begin{definition}{\rm
The {\it radical} of a congruence $I$ is the intersection of all prime congruences containing $I$. It is denoted by $Rad(I)$. A congruence $I$ is called a {\it radical congruence} if $Rad(I)=I$.
}
\end{definition}

Let us introduce the following notation: for a pair $\alpha$, let $\alpha^* = (\alpha_1+\alpha_2,0)$.
It is easy to verify the following proposition:

\begin{proposition} \label{prop: *basic}
Let $\alpha,\beta \in A$ pairs from the $\mb$-algebra A, 
\begin{itemize}
\item[(i)] $(\alpha\beta)^* = \alpha^*\beta^*$
\item[(ii)] $((\alpha\beta)^*)^k = ((\alpha\beta)^k)^*$
\item[(iii)] If $\alpha^*  \in \diag$ then $\alpha \in \diag$.

\end{itemize}
\end{proposition}

Now we will define a property for pairs in $A \times A$ that is analogous to nilpotency from ring theory. The aim of this section is to show that the pairs contained in every prime congruence are precisely the nilpotent ones. A natural first guess would be to define the pair $\alpha$ to be nilpotent if $\alpha^n \in \diag$ for some $n$. Indeed, in the case of commutative rings, one could characterize the congruence with kernel the nilradical in this fashion. However as shown by the following example these pairs do not even form a congruence in the case of $\mb$-algebras:

\begin{example}{\rm
In the three variable polynomial semiring $\mb[x_1,x_2,x_3]$ take the congruence $I = \<(x_1,x_2)^2,(x_2,x_3)^2\>$. 
Since $(x_1,x_2)^2 = (x_1^2+x_2^2, x_1x_2)$ and $(x_2,x_3)^2 = (x_2^2+x_3^2, x_2x_3)$ one easily verifies that any pair in $I \setminus \diag$ will need to contain a monomial divisible by $x_2$ on both sides, hence we have $(x_1,x_3)^k \notin I$ for any $k>0$.
It follows that in the quotient $\mb[x_1,x_2,x_3] / I$ the pairs $\alpha$ that satisfy $\alpha^k \in \diag$ for some $k$ do not form a congruence, since otherwise $(x_1,x_3)$ would have to be amongst them by transitivity. }
\end{example}

To remedy this problem we will introduce some formulas, called generalized powers of pairs that will turn out to have the desired properties.

\begin{definition}{\rm
For a pair $\alpha$ from the $\mb$-algebra $A$, the {\it \genps}of $\alpha$ are the pairs of the form $({\alpha^*}^k+(c,0))\alpha^l$ where $k,l$ are non-negative integers, and $c\in A$ an arbitrary element. The set of \genps  of $\alpha$ is denoted by $\gp(\alpha)$. A pair $\alpha$ is called {\it nilpotent} if  $\gp(\alpha) \cap \diag \neq \emptyset$.
}
\end{definition}

\begin{proposition}
For an arbitrary pair $\alpha$ the set $\gp(\alpha)$ is closed under twisted product. Moreover, if $\beta \in \gp(\alpha)$ then $\gp(\beta) \subseteq \gp(\alpha)$.
\end{proposition}
\begin{proof}
Both claims follow directly from the definition and Proposition \ref{prop: *basic}.
\end{proof}

One can immediately show the following:

\begin{proposition}\label{prop: Nilpotents in radical}
The nilpotent pairs are contained in every prime congruence.
\end{proposition}
\begin{proof}
Indeed if $({\alpha^*}^k+(c,0))\alpha^l \in \diag$ then for any prime congruence $P$ we have that $({\alpha^*}^k+(c,0))\alpha^l \in P$, which implies that either $\alpha \in P$ or $({\alpha^*}^k+(c,0))\in P$. Moreover, if $({\alpha^*}^k+(c,0))\in P$ then by (ii) in Proposition \ref{prop: congbasic} we have that ${\alpha^*}^k \in P$ and by Proposition \ref{prop: *basic} $\alpha^* = (\alpha_1+\alpha_2,0) \in P$, now applying (i) from Proposition \ref{prop: congbasic} we get that $(\alpha_1,0) \in P$ and $(\alpha_2,0) \in P$ so $\alpha \in P$.
\end{proof}

Now we prepare to show that the reverse implication holds as well. We need the following two lemmas:

\begin{lemma}\label{lemma: xcong0}
Let $x \in A$ be an arbitrary element and $I = \<(x,0)\>$. Then $(y,z) \in I$ if and only if there exist an $r\in A$ such that $y+rx = z + rx$.
\end{lemma}

\begin{proof}
Let $J$ be the set of pairs $(y,z)$ such that  there exist an $r\in A$ such that $y+rx = z + rx$. Clearly $(x,0) \in J$ and $J \subseteq I$, so it is enough to show that $J$ is a congruence.  C1 and C2 hold trivially. For C3 assume that $y+rx = z + rx$ and $z+sx = v + sx$, then we have $y+(r+s)x = z + (r+s)x = v+(r+s)x$ giving us $(y,v) \in J$. For C4 and C5 assume that $y+rx = z + rx$ and $v+rx = w + rx$ then we have $y+v + (r+s)x = v+w + (r+s)x$ and $yv + (vr+zs)x = zv + (vr+zs)x = zw + (vr+zs)x$ showing that both conditions hold. 
\end{proof}

\begin{lemma}\label{lem: genps in diag}
If for some $c,x \in A$ and a pair $\alpha$ from $A$ we have that  \[({\alpha^*}+(c,0))\alpha \in \<(x,0)\> \cap Ann(x)\] then there exists a $b\in A$ such that $({\alpha^*}^3+(b,0))\alpha \in \diag$.
\end{lemma}

\begin{proof}
Since  $({\alpha^*}+(c,0))\alpha \in \<(x,0)\>$ by Lemma \ref{lemma: xcong0} we have that for some $r \in A$ \[\alpha_1^2+\alpha_1\alpha_2+c\alpha_1 + rx = \alpha_2^2+\alpha_1\alpha_2+c\alpha_2 + rx \]
Let $y = rx$. By $({\alpha^*}+(c,0))\alpha \in Ann(x)$ we have that \[y(\alpha_1^2+\alpha_1\alpha_2+c\alpha_1) = y(\alpha_2^2+\alpha_1\alpha_2+c\alpha_2). \]
Set $b=y(\alpha_1 + \alpha_2 + c) + c(\alpha_1+\alpha_2)^2$, and $\beta = ({\alpha^*}^3+(b,0))\alpha$. After expanding we get: \[\beta_1 = \sum_{i=1}^{4}\alpha_1^i\alpha_2^{(4-i)}+y(\alpha_1^2 + \alpha_1\alpha_2+c\alpha_1)+c( \sum_{i=1}^{3}\alpha_1^i\alpha_2^{(3-i)})\] \[\beta_2 = \sum_{i=1}^{4}\alpha_2^i\alpha_1^{(4-i)}+y(\alpha_2^2 + \alpha_1\alpha_2+c\alpha_2)+c( \sum_{i=1}^{3}\alpha_2^i\alpha_1^{(3-i)})\]
The terms appearing in $\beta_2$ but not in $\beta_1$ are $\alpha_2^4,\:y\alpha_2^2,\:yc\alpha_2,\:c\alpha_2^3$. However we have: \[\beta_1 \geq y(\alpha_1^2 + \alpha_1\alpha_2+c\alpha_1) = y(\alpha_2^2 + \alpha_1\alpha_2+c\alpha_2) \geq y\alpha_2^2 + yc\alpha_2\] 
It follows that \[\beta_2 \geq \alpha_2^2(\alpha_1^2 +\alpha_1\alpha_2 + c\alpha_1 + y) =  \alpha_2^2(\alpha_2^2 +\alpha_1\alpha_2 + c\alpha_2 + y) \geq \alpha_2^4 + c\alpha_2^3 \]
showing us $\beta_1 \geq \beta_2$ and by symmetry $\beta_1 = \beta_2$, so indeed $\beta \in \diag$.
\end{proof}

We are ready to prove:

\begin{theorem}\label{thm:radical}
For any congruence $I$ of a $\mb$-algebra A, we have that \[Rad(I) = \{\alpha \mid \gp(\alpha) \cap I \neq \emptyset \}.\] In particular the intersection of every prime congruence of $A$ is precisely the set of nilpotent pairs.
\end{theorem}
\begin{proof}
Note that the intersection of all prime congruences is $Rad(\diag)$. We can reduce to the case $I=\diag$ after considering the quotient $A/I$. Proposition \ref{prop: Nilpotents in radical} tells us that the nilpotent elements are contained in $Rad(\diag)$, for the other direction we have to show that for a non-nilpotent pair $\alpha$ there is a prime congruence $P$ such that $\alpha \notin P$. We have that $\gp(\alpha) \cap \diag = \emptyset$. By Zorn's lemma there is a congruence $J$ that is maximal amongst the congruences that are disjoint from $\gp(\alpha)$. If $J$ is prime we are done. Assume $J$ is not prime, we first show that $J$ is intersection indecomposable. Assume the contrary $J = K \cap L$ for some congruences $J \subsetneq K,L$. Then the maximality of $J$ implies that there exists a $\beta \in K \cap \gp(\alpha)$ and a  $\gamma \in L \cap \gp(\alpha)$, but then $\beta\gamma \in L\cap K \cap \gp(\alpha) = J \cap \gp(\alpha)$ a contradiction. So J is not prime but intersection indecomposable, then it follows from Theorem \ref{thm: ind QC are prime} that $J$ is not QC.  Thus there exists a non-zero $x \in A/J$ such that $Ann_{A/J}(x) \supset  \diag_{A/J}$.  Let $K$ be the congruence generated by $(x,0)$ in $A/J$.  Again by maximality, we have that every non-trivial congruence in $A/J$ contains some element of $\gp(\alpha)$, so in particular for some $k,l,c$ we have an element $({\alpha^*}^k+(c,0))\alpha^l \in \gp(\alpha) \cap Ann_{A/J}(x) \cap K$. After multiplying with some power of $\alpha^*$ or $\alpha$ (depending on which of $k$ or $l$ is larger) we can assume that $k=l$. Now we can apply Lemma \ref{lem: genps in diag} for the pair $\alpha^k$ and the semiring $A/J$ and obtain that for some $b$ we have $({\alpha^*}^{3k}+(b,0))\alpha^k \in J$ contradicting $\gp(\alpha) \cap J = \emptyset$.
\end{proof}

We conclude this section by a list of corollaries of the above theorem.

\begin{proposition}
QC congruences are radical congruences.
\end{proposition}
\begin{proof}
By considering the appropriate quotients it is enough to prove the theorem for the case when the congruence is the trivial congruence. We have to show that if for some pair $\alpha$ we have $\gp(\alpha) \cap \diag \neq \emptyset$ then $\alpha \in \diag$. Suppose that for some $k,l$ we have  $({\alpha^*}^k+(c,0))\alpha^l \in \diag$. Then by cancellativity either $\alpha^l\in \diag$ and then by Lemma \ref{cancellative powers} $\alpha \in \diag$, or $({\alpha^*}^k+(c,0)) \in \diag$ and then from Proposition \ref{prop: congbasic} it follows that ${\alpha^*}^k \in \diag$  which in turn by Proposition \ref{prop: *basic} implies that $\alpha^k \in \diag$, and finally by Lemma \ref{cancellative powers} that $\alpha \in \diag$.
\end{proof}

Let us denote by $\gann_A(\alpha)$ the set $\{\beta\mid \gp(\alpha\beta) \cap \diag \neq \emptyset\}$.

\begin{proposition}\label{prop:rad corol1}
Let A be an arbitrary  $\mb$-algebra and $\alpha \in A\times A$ a pair. 
\begin{itemize}
\item[(i)] $\gann_A(\alpha)$ is the intersection of all prime congruences not containing $\alpha$ (where by empty intersection we mean the full set $A \times A$), in particular $\gann_A(\alpha)$ is a congruence.
\item[(ii)] If $\diag$ is a radical congruence then $Ann_A(\alpha) = \gann_A(\alpha)$, in particular $Ann_A(\alpha)$ is a congruence.
\end{itemize}
\end{proposition}
\begin{proof}
First let $\beta \in \gann_A(\alpha)$. Then by Theorem \ref{thm:radical}, we have that $\alpha\beta \in Rad(\diag) = \bigcap_{P\:prime}P$, so by the prime property every prime that does not contain $\alpha$ needs to contain $\beta$. For the other direction let $\beta$ be an element of every prime congruence that does not contain $\alpha$, then $\alpha\beta$ is contained in every prime and by Theorem \ref{thm:radical} $\gp(\alpha\beta) \cap \diag \neq \emptyset$. The second half of the statement follows from the fact that if $\diag$ is a radical congruence then $\gp(\alpha\beta) \cap \diag \neq \emptyset$ implies $\alpha\beta \in \diag$.
\end{proof}

While it might appear that Proposition \ref{prop:rad corol1} provides a simpler proof for Lemma \ref{lem:cancellative annulator} and Theorem \ref{thm: ind QC are prime}, but we remind the reader that  Theorem \ref{thm: ind QC are prime} was used in the proof of Theorem \ref{thm:radical} which in turn we used to prove Proposition \ref{prop:rad corol1}.

\begin{proposition}\label{prop: radindec}
A congruence is prime if and only if it is radical and intersection indecomposable.
\end{proposition}
\begin{proof}
Prime congruences are radical by definition and intersection indecomposable by Proposition  \ref{prop:primes are indec}. For the other direction we can argue the same way as in the proof of Theorem \ref{thm: ind QC are prime}, except that this time $\beta^2 \in \diag$ implies $\beta \in \diag$ simply by the definition of a radical congruence.
\end{proof}

\subsection{Semialgebras satisfying the ACC}\label{sec:ACC}

While most of the algebras in this paper do not satisfy the ascending chain condition (ACC) for congruences, we make a few remarks about the ones that do satisfy it. Firstly, we have the following statement from ring theory that holds in this setting. The argument for it is essentially the same as in the classical case.

\begin{proposition}\label{prop: finiteminprimes}
Let $A$ be a  $\mb$-algebra with no infinite ascending chain of radical congruences. Then over every congruence there are finitely many minimal primes.
\end{proposition}
\begin{proof}
The primes lying over a congruence $I$ are the same as the primes lying over $Rad(I)$, so it is enough to prove the statement for radical congruences.  Assume that there are radical congruences of $A$ with infinitely many minimal primes lying over them, and let $J$ be a maximal congruence amongst these. Since $J$ is not prime then by Proposition \ref{prop: radindec} it is the intersection of two strictly larger congruences $K$ and $L$. Then every prime containing $J$ contains at least one of $K$ and $L$ so the minimal primes lying over $J$ are amongst those that are minimal over $K$ or $L$ and by the maximality of $J$ there is only finitely many of these. 
\end{proof}

One can define primary congruences in the following way:

\begin{definition}{\rm
We will call a congruence $I$ of a $\mb$-algebra $A$ primary if $\{\alpha\mid \exists \beta\notin I :\alpha\beta \in I \}\subseteq Rad(I)$.
}
\end{definition}

As one would expect this class satisfies the following property:

\begin{proposition}
The radical of a primary congruence is a prime congruence.
\end{proposition}
\begin{proof}
Let $Q$ be a primary congruence, assume that $Rad(Q)$ is not prime. Then we have $\alpha,\beta \notin Rad(Q)$ such that $\alpha\beta \in Rad(Q)$. Then for some $k,l$ we have  $({(\alpha\beta)^*}^k+(c,0))(\alpha\beta)^l \in Q$. Now since $GP(\alpha^l) \subseteq GP(\alpha)$, neither $\alpha^l$ nor $\beta^l$ can be in $Rad(Q)$ so by the primary property we have that $({(\alpha\beta)^*}^k+(c,0))\in Q$ implying ${(\alpha\beta)^*}^k \in Q$. Since ${(\alpha\beta)^*}^k = {(\alpha^*)}^k{(\beta^*)}^k$, this means that at least one of $\alpha^*$, $\beta^*$ is nilpotent in the quotient by $Q$, but then since $GP(\alpha^*) \subseteq GP(\alpha)$ we have that $\alpha$ or $\beta$ is nilpotent, a contradiction.
\end{proof}

Unfortunately, there is no general analogue of primary decomposition from commutative algebra. It is easy to show an example of an intersection indecomposable congruence that is not primary in a semiring that satisfies the ACC.

\begin{example}{\rm
Consider the 4-element  $\mb$-algebra $A$, with set of elements $\{1,0,x,y\}$ satisfying the relations $\{1+x=1,x+y=x,x^2=x,xy=0,y^2 = 0\}$. It is easy to check that the 3 non-trivial proper congruences of this algebra are $I_1 = \{(0,y)\}$ $I_2 = \{(0,y),(0,x)\}$ $I_3 = \{(0,y),(1,x)\}$. We see that $I_1 \subseteq I_2,I_3$ so $\diag$ is intersection indecomposable. $A/I_2 \cong \mb$ and $A/I_3 \cong \mb$ so $I_2$ and $I_3$ are prime congruences. Also we have that $(1,x)(x,0) = (x,x) \in \diag$, so neither $I_1$ nor the trivial congruence are prime. It follows that $Rad(\diag) = I_2\cap I_3=I_1$ and $(1,x) \notin Rad(\diag)$ so $\diag$ is intersection indecomposable but not primary. Also note that $Rad(\diag)$ in this case is not prime so even if one changes the notion of primary congruences, as long as we require the radical of primaries to be primes this algebra would provide a counterexample to primary decomposition.
}
\end{example}

\section{Prime congruences of polynomial and Laurent polynomial semirings}\label{sec:polysemirings}
\subsection{The prime congruences of $\mb(\bx)$ and $\mb[\bx]$}\label{sec:b[x]}
Throughout this section $\mb(\bx)$ and $\mb[\bx]$ denote the Laurent polynomial semiring and the polynomial semiring with $k$ variables $\bx = (x_1,\dots,x_k)$. First we show that the kernel of the primes of these semirings are easy to describe:

\begin{proposition}\label{prop:kernels}
\begin{itemize}
\item[(i)] For any proper congruence $I$ of $\mb(\bx)$, we have that $Ker(I) = \{0\}$.
\item[(ii)] For any QC congruence $Q$ of $\mb[\bx]$ we have that $Ker(Q)$ is the polynomial semialgebra generated by a subset of the variables $x_1,\dots,x_k$.
\end{itemize}
\end{proposition}
\begin{proof}
In both cases by Proposition \ref{prop: congbasic} we have that the kernel of any congruence is generated by monomials. In the case of $\mb(\bx)$ any monomial has a multiplicative inverse, so if $Ker(I) \neq \{0\}$ then we have $(1,0) \in Ker(I)$ so $I$ has to be the improper congruence. For (ii) if $Q$ is QC then $(fg,0) \in Q$ implies that $(f,0) \in Q$ or $(g,0) \in Q$, so a monomial is in $Ker(Q)$ if and only if at least one of the variables in that monomial is in $Ker(Q)$. 
\end{proof}

So in fact prime congruences of $\mb[\bx]$ with non-zero kernels will correspond to prime congruences of a polynomial semirings in less variables. Next recall that quotients by primes are totally ordered and consider the following proposition:

\begin{proposition}\label{prop: orderingbasic}
\begin{itemize}
\item[(i)] If $Q$ is a congruence of  $\mb[\bx]$  or $\mb(\bx)$ such that the quotient by $Q$ is totally ordered, then in each equivalence class of $Q$ there is at least one monomial. 
\item[(ii)] A congruence $P$ of $\mb(\bx)$ is prime if and only if $\mb(\bx)/P$ is totally ordered.
\item[(iii)] If $Q$ is a prime congruence of $\mb[\bx]$ with $Ker(Q) = \{0\}$, then $Q = P|_{\mb[\bx]}=P$ for some prime congruence $P$ of $\mb(\bx)$.
\item[(iv)] For a prime $P$ of $\mb(\bx)$ the multiplicative monoid of $\mb(\bx)/P$ is isomorphic to a quotient of the additive group $(\mz^k,+)$. For a prime $P$ of $\mb[\bx]$ the multiplicative monoid of $\mb[\bx]/P$ is isomorphic to the restriction of a quotient of the additive group $(\mz^{k'},+)$ to $(\mn^{k'},+)$, where $k-k'=|\{x_1,\dots,x_k\}\cap Ker(P)|$.
\end{itemize}
\end{proposition}
\begin{proof}
The first statement follows from the fact that if the quotient is totally ordered, then every polynomial is congruent to any of its monomials that is maximal with respect to the ordering on the quotient. For (ii) consider that every monomial in  $\mb(\bx)$ has a multiplicative inverse, so by (i) we see that the if the quotient by a congruence $P$ is totally ordered then it is a semifield, which is in particular cancellative and then by Proposition \ref{prop:primeorder} $P$ is prime. For (iii) first note that congruences of $\mb(\bx)$ with totally ordered quotients are determined by the equivalence class of $1$. Take a prime congruence $Q$ of $\mb[\bx]$ with $Ker(Q) = \{0\}$, and let $P$ be the congruence of  $\mb(\bx)$ with a totally ordered quotient satisfying that for any monomials $m_1,m_2 \in \mb[\bx]$: $$(1,m_1/m_2) \in P \iff (m_2,m_1) \in Q\;\mbox{and}\;(1,m_1/m_2 + 1) \in P \iff (m_2,m_1 + m_2) \in Q.$$ Note that while writing a Laurent monomial as quotient of monomials of $\mb[\bx]$ is not done uniquely, the above is still well defined because of the QC property of $Q$. $P$ is prime since its quotient is totally ordered and cancellative and it is straightforward to check that $P|_{\mb[\bx]} = Q$. (iv) follows from (i),(iii) and Proposition \ref{prop:kernels}.
\end{proof}

A {\it group ordering} (resp. {\it semigroup ordering}) of a group (resp. semigroup) $(G,+)$, is an ordering $\leq$ on the elements of $G$ satisfying that for any $g_1,g_2 \in G$ with $g_1 \leq g_2$ and an arbitrary $g_3 \in G$ we have $g_1+g_3 \leq g_2+g_3$. The previous proposition tells us that to understand the prime quotients of $\mb(\bx)$ we need to describe the group orderings on the quotients of $(\mz^k,+)$. When we think of $(\mz^k,+)$ (resp. $(\mn^k,+)$) as the group (resp. semigroup) of Laurent monomials (resp. monomials) with the usual multiplication their group orderings are called {\it term orderings}. (Note that in the literature it is sometimes required that the generating variables are larger than the unit under a term ordering, but we do not use this convention). Term orderings are described by a result of Robbiano in \cite{Rob85}:

\begin{proposition}\label{prop: Robbiano}
For every term ordering $\leq$ of the Laurent monomials $\{\bx^\bn\mid \bn \in \mz^k\}$ there exist a matrix $U$  with $k$ columns and $l\leq k$ rows, such that $\bx^{\bn_1} <  \bx^{\bn_2}$ if and only if the first non-zero coordinate of $U(\bn_2-\bn_1)$  is positive. Term orderings of the monomials $\{\bx^\bn\mid \bn \in \mn^k\}$ are restrictions of the orderings on the Laurent monomials.
\end{proposition}

We will say that the $i$-th row of the matrix $U$ is {\it non-redundant} if there is an integer vector $\bn \in \mz^k$ such that the first non-zero coordinate of $U\bn$ is the $i$-th coordinate. If all of the rows of $U$ are non-redundant we will call it an {\it admissible} matrix. If $U$ is an admissible matrix for an ordering as in the setting of Proposition \ref{prop: Robbiano}, then it will be called a defining matrix of the ordering. It is easy to verify that the defining matrix can always be chosen to have orthonormal rows, and that for an ordering defined by a square matrix there is a unique orthogonal defining matrix.As explained above, term orderings define prime congruences of $\mb(\bx)$ and $\mb[\bx]$, which will be denoted by $P(U)$ and $P[U]$ respectively. One can also consider the $\mb$-algebra of Laurent monomials (resp. monomials) whose addition is defined by the term ordering of $U$, and the surjections from $\mb(\bx)$ (resp. $\mb[\bx]$) onto these that map each polynomial to their leading monomial, then $P(U)$ (resp. $P[U]$) are just the kernel of these maps. Note that prime congruences given by term orderings are minimal by (i) of Proposition \ref{prop: orderingbasic} since every equivalence class of them contains precisely one monomial. \par\smallskip
If an admissible matrix $U$ is the defining matrix of a term ordering then the zero vector is the only integer vector in the kernel of $U$, since a term ordering is a total ordering of all of the monomials. If $U$ has integer vectors in its kernel, it still gives us a group ordering on the quotient $\mz^k/(Ker(U)\cap \mz^k)$, defined the same way as in Proposition \ref{prop: Robbiano}. In this case we will still call $U$ the defining matrix of the ordering on that quotient and denote by $P(U)$ or $P[U]$ the corresponding prime congruences of $\mb(\bx)$ and $\mb[\bx]$. Explicitly speaking, $P(U)$ is generated by the pairs $(\bx^{\bn_1}+\bx^{\bn_2},\bx^{\bn_2})$ such that either $U(\bn_2-\bn_1)=\b0$ or the first non-zero coordinate of $U(\bn_1-\bn_2)$ is positive and $P[U] = P(U)|_{\mb[\bx]}$. We will soon see that every prime congruence of these $\mb$-algebras arise this way. \par\smallskip
Since the rows of an admissible matrix $U$ are linearly independent its rank $r(U)$ is equal to the number of its rows. For $i\leq r=r(U)$ let us denote by $U(i)$ the matrix that consists of the first $i$ rows of $U$. Note that if $U$ is admissible then so are all of the $U(i)$. Let us use the convention that $U(0)$ for any $U$ is the "empty matrix" which corresponds to the only group ordering of the one element quotient $\mz^k/\mz^k$ and $P(U(0))$ (resp. $P[U(0)]$) are the maximal congruences of $\mb(\bx)$ (resp. $\mb[\bx]$) that identify every non-zero element with $1$. Accordingly we will write $r(U(0)) = 0$.
Now we describe the primes lying above a congruence $P(U)$.

\begin{proposition}\label{prop: primes of term orderings}
Let $U$ be an admissible matrix with $k$ columns. Then every proper congruence of $\mb(\bx)$ containing $P(U)$ is an element of the strictly increasing chain $$P(U)=P(U(r(U)))\subset P(U(r(U)-1)) \subset \dots\subset P((U(0))).$$ In particular every proper congruence of $\mb(\bx)/P(U)$ is prime and $dim(\mb(\bx)/P(U)) = r(U)$.
\end{proposition}
\begin{proof}
The congruences $P(U(i))$ are prime since their quotients are totally ordered and cancellative. Furthermore, the chain in the proposition is strictly increasing since the rows of $U$ are non-redundant. Since the $P(U(i))$-s form a finite chain, it is enough to verify that every congruence that is generated by a single pair is one of these, and then it will follow for an arbitrary congruence $P(U)\subseteq I$ that $I=P(U(i))$ where $i$ is the smallest such that $P(U(i))$ can be generated by a pair in $I$. Note that in a semifield each congruence is determined by the equivalence class of $1$, since for any congruence $I$ we have that $(\alpha_1,\alpha_2) \in I \iff (\alpha_1\alpha_2^{-1},1)\in I$. Therefore for any congruence $P(U) \subsetneq I$ generated by a single pair we have that $I = \<(1,\bx^\bn)\>$ for some $\bn \in \mz^k$ satisfying $\bn \notin Ker(U)$. Let $s$ be the smallest integer such that for the $s$-entry of $U\bn$ we have $(U\bn)[s] \neq 0$, then we have that $(1,\bx^\bn) \in  P(U(s-1))$. Moreover, if $(1,\bx^{\bn'}) \in  P(U(s-1))$ for some $\bn'$, then $\forall j<s:\; (U\bn')[j] = 0$. Then for some $k \in \mz$ with large enough absolute value we have that either $1 \leq \bx^{\bn'} \leq \bx^{k\bn}$ or $ \bx^{k\bn} \leq \bx^{\bn'} \leq 1$ where $\leq$ is the ordering on the quotient $\mb(\bx)/P(U)$. Then by (iii) of  Proposition \ref{prop: congbasic} we have that $(1, \bx^{\bn'}) \in I$, so $P(U(s-1)) \subseteq I$ and then $P(U(s-1))  = I$.
\end{proof}

Finally, we need the following lemma to prove our main result:

\begin{lemma}\label{lem: minprimes}
For every prime congruence $Q$ of $\mb(\bx)$ we have an admissible matrix $U$ such that $P(U) \subseteq Q$ and $Ker(U)\cap\mz^k = \{\b0\}$.
\end{lemma}
\begin{proof}
Recall that for an admissible matrix $U$ the condition $Ker(U)\cap\mz^k = \{\b0\}$ is equivalent to saying that $U$ is the defining matrix of a term ordering. Intuitively speaking $U$ can be obtained by taking an arbitrary ordering on the subspace that $Q$ identifies with $1$. To see this, denote the ordering induced by the addition on $\mb(\bx)/Q$ by $\leq_Q$ and fix an arbitrary term ordering $\preceq_0$. Now we define a new term ordering $\preceq$ as $$m_1 \preceq m_2 \iff m_1 <_Q m_2 \;or\;[(m_1,m_2)\in Q\;and\;m_1\preceq_0 m_2].$$
To verify that $\preceq$ is indeed a term ordering consider $m_1,m_2$ such that $m_1 \preceq m_2$ and an arbitrary monomial $s \neq 0$. We have that either $m_1 <_Q m_2$, but then by the cancellativity of  $\mb(\bx)/Q$ it follows that $sm_1 <_Q sm_2$, or $(m_1,m_2) \in Q$ and $m_1 \preceq_0 m_2$ and then since $Q$ is a congruence and $\preceq_0$ is a term ordering we have that  $(sm_1,sm_2) \in Q$ and $sm_1 \preceq_0 sm_2$. Now from the definition of $\preceq$ we see that $m_1 \preceq m_2 \Rightarrow m_1 \leq_Q m_2$, so for the defining matrix $U$ of $\preceq$ we have $P(U) \subseteq Q$.
\end{proof}

A {\it lattice polytope} in $\mr^k$ is just a polytope whose vertices are all in $\mz^k$. The {\it Newton polytope} of a polynomial $f = \sum_i \bx^{\bn_i}$ of $\mb(\bx)$ or $\mb[\bx]$ is the convex hull of the lattice points $\bn_i \in \mz^k$. It will be denoted by $newt(f)$. By convention $newt(0)$ is the empty set.
Now we proceed to describe the prime congruences and radical of $\mb(\bx)$. We remind that by convention we also write the maximal congruence of $\mb(\bx)$ as $P(U)$ where $U$ is a matrix with "zero rows".

\begin{theorem}\label{thm:primesofb(x)} For the $k$-variable Laurent polynomial semialgebra $\mb(\bx)$ we have that:
\begin{itemize}
\item[(i)] The set of prime congruences of  $\mb(\bx)$ is $\{P(U)\mid\:U\; is \; an\; admissible \; matrix \; with \; k\; columns\}$. The prime congruence $P(U)$ is minimal if and only if $Ker(U)\cap\mz^k = \{\b0\}$.
\item[(ii)] $dim(\mb(\bx)) = k$. 
\item[(iii)] The pair $(f,g)$ lies in the radical of the trivial congruence of $\mb(\bx)$ if and only if $newt(f) = newt(g)$.
\item[(iv)] The $\mb$-algebra $\mb(\bx)/Rad(\diag)$ is isomorphic to the  $\mb$-algebra with elements the lattice polytopes and addition being defined as the convex hull of the union, and multiplication as the Minkowski sum.
\item[(v)] Every radical congruence is QC.
\end{itemize}
\end{theorem}

\begin{proof}
For (i) consider that by Lemma \ref{lem: minprimes} every prime contains a prime $P(U)$ with $Ker(U)\cap\mz^k = \{\b0\}$ and by Proposition \ref{prop: primes of term orderings} every prime lying over some $P(U)$ is $P(U(i))$ for some $0\leq i \leq r(U)$.  (ii) follows from Proposition \ref{prop: primes of term orderings} and the fact that there are term orderings whose defining series is of length $k$ (for example the usual lexicographic order). For (iii) first note that since every prime is contained in a minimal prime the radical of the trivial congruence is the intersection of the minimal primes. By (i) a minimal prime $P(U)$ corresponds to a term ordering, and for a monomial $m$ and a polynomial $f$ we have $(f,m) \in P(U)$ if and only if $m$ is the leading term of $f$ in the corresponding term ordering. Hence it is enough to show that the set of vertices of $newt(f)$ are precisely the exponents of the monomials of $f$ that are leading terms with respect to some term ordering. On one hand by Proposition \ref{prop: Robbiano} the leading term is determined by maximizing a set of linear functionals on $newt(f)$, so its exponent indeed has to be one of the vertices. On the other hand for any vertex $v$ of $newt(f)$ one can pick a hyperplane that separates it from the rest of the vertices. Choosing the normal vector $\bu$ of such a hyperplane to point towards the side of $v$, for any admissible matrix $U$ with $Ker(U)\cap\mz^k = \{\b0\}$ having $\bu$ as a first row we have that the leading term of $f$ in the term ordering defined by $U$ is the monomial with exponent $v$. Now since the set of vertices determine the polytope $newt(f)$ we have that $(f,g)$ lies in every prime if and only if $newt(f) = newt(g)$. For (iv) one easily checks that $newt(f+g)$ is the convex hull of $newt(f) \cup newt(g)$ and $newt(fg)$ is the Minkowski sum of $newt(f)$ and $newt(g)$. For (v) assume that for a radical congruence $I$,  $(g,0)(f_1,f_2) \in I$ then $(g,0)(f_1,f_2)$ is in every prime containing $I$, but since all primes have trivial kernels $(f_1,f_2)$ has to be in every prime containing $I$ and then $(f_1,f_2) \in I$.
\end{proof}

In the one variable case there are two different term orderings, and for two or more variables there are infinitely many, hence by Proposition \ref{prop: finiteminprimes} we have the following corollary:

\begin{corollary}\label{cor: infinitechains}
If $k>1$ there are infinitely many minimal prime congruences in $\mb(\bx)$ and if $k = 1$ there are exactly two. In particular for $k>1$ $\mb(\bx)$ does not satisfy the ACC for radical congruences (or equivalently for QC congruences).
\end{corollary}

Now we turn to $\mb[\bx]$. Recall from (iii) of Proposition \ref{prop: orderingbasic} that the primes of $\mb[\bx]$ with trivial kernel are restrictions of the primes of $\mb(\bx)$. Here we also have over any prime $P[U]$ the strictly increasing chain $$P[U]=P[U(r(U))]\subset P[U(r(U)-1)] \subset \dots\subset P[U(0)].$$
It follows that $dim(P[U]) \geq dim(P(U)) = r(U)$,  the next proposition shows that the dimensions are in fact equal.

\begin{proposition}\label{prop:dimofP[U]}
For any admissible matrix $U$ we have that $dim(\mb[\bx]/P[U]) = r(U)$.
\end{proposition}
\begin{proof}
We will prove by induction on $r(U)$. The $r(U) = 0$ case is clear, since by our earlier conventions for the matrix with "zero rows" we have $\mb[\bx]/P[U] = \mb$ and $dim(\mb) = 0$. Let $U$ now be an arbitrary admissible matrix and $Q$ a prime congruence that is minimal amongst those that strictly contain $P[U]$, to complete the proof we need to show that $dim(\mb[\bx]/Q) \leq r(U)-1$. If $Ker(Q) = \{0\}$ then by (iii) of Proposition \ref{prop: orderingbasic} and Proposition \ref{prop: primes of term orderings} we have that $Ker(Q) = P[U(r(U)-1)]$ and then by the induction hypothesis we have $dim(\mb[\bx]/Q) = r(U)-1$. If $Ker(Q) \neq \{0\}$ then by Proposition \ref{prop:kernels}, $Ker(Q)$ is generated by a subset of the variables, say $x_1,\dots,x_j$. Also by the minimality of $Q$ we have that $Q = \<P(U)\cup\{(x_i,0)|1\leq i \leq j\}\>$. It follows that for some prime $P[U_Q]$ of $\mb[x_{j+1},\dots,x_k]$ the quotient $\mb[\bx]/Q$ is isomorphic to $\mb[x_{j+1},\dots,x_k]/P[U_Q]$. The matrix $U_Q$ can be obtained from $U$ by removing the first $j$ columns, then removing any possible redundant rows. Now since $(1,0) \notin Q$ by (iii) of Proposition \ref{prop: congbasic} we have that for any monomial $m$ containing any of the variables $x_1,\dots,x_j$, $m < 1$ in the ordering defined by $U$. This implies that the for some $1\leq i \leq r(U)$ the first $i$ rows of $U$ have to be such that all non-zero entries are in the first $j$ columns, and the first non-zero entry in those columns is negative. Consequently when the first $j$ columns are removed from $U$, then the first $i$ rows will have all $0$-s as the remaining entries, so they are removed when we obtain $U_Q$. In particular we have that $dim(\mb[\bx]/Q) = r(U_Q) < r(U)$ completing the proof.
\end{proof} 

Now we have the following theorem about the primes and radical of $\mb[\bx]$:

\begin{theorem}\label{thm:primesofb[x]} For the $k$-variable polynomial semiring $\mb[\bx]$ we have that,
\begin{itemize}
\item[(i)] For every prime congruence $P$ of  $\mb[\bx]$ there is a (possibly empty) subset $H$ of the variables $\bx$ and a prime $P[U]$ of the polynomial semiring $\mb[\bx']$ with variables $\bx'=\bx\setminus H$, such that  $P$ is generated by the pairs $\{(x_i,0)|\:x_i\in H\}$ and the image of $P[U]$ under the embedding $\mb[\bx']\hookrightarrow \mb[\bx]$.
\item[(ii)] The minimal prime congruences of $\mb[\bx]$ have $\{0\}$ as their kernel and are all of the form $P[U]$, where $U$ is an admissible matrix with $Ker(U)\cap\mz^k = \{\b0\}$.
\item[(iii)] $dim(\mb[\bx]) = k$. 
\item[(iv)] The pair $(f,g)$ lies in the radical of the trivial congruence of $\mb[\bx]$ if and only if $newt(f) = newt(g)$.
\item[(v)] The $\mb$-algebra $\mb[\bx]/Rad(\diag)$ is isomorphic to the  $\mb$-algebra with elements the lattice polytopes lying in the non negative quadrant $\mr_{+,0}^k$, and addition being defined as the convex hull of the union, and multiplication as the Minkowski sum.
\item[(vi)] The congruence $Rad(\diag)$ is QC.
\end{itemize}
\end{theorem}
\begin{proof}
(i) follows from Proposition \ref{prop:kernels}, Theorem \ref{thm:primesofb(x)} and (iii) of Proposition \ref{prop: orderingbasic}. For (ii) let $Q$ be a minimal prime congruence with $Ker(Q) \neq 0$. We can assume that $Ker(Q)$ is generated by the variables $x_1,\dots,x_j$ for some $j$. By the minimality of $Q$, $\mb[\bx]/Q$ is isomorphic to $P[U']$ where $U'$ is the defining matrix of a term ordering on the variables $x_{j+1},\dots,x_k$. Let $U$ be the defining matrix of the term ordering that first orders the variables $x_1,\dots,x_j$ reverse lexicographically, then the rest of the variables by $U'$ (so the first $j$ rows of $U$ are negatives of the first $j$ rows of the identity matrix). Now for the prime congruence $P[U]$ we have $Ker(P[U]) = \{0\}$ and $P[U] \subseteq Q$. (iii) follows from (ii) and Proposition \ref{prop:dimofP[U]}. (iv) and (v) follow by the same argument as in the proof of Theorem \ref{thm:primesofb(x)}. Finally, (vi) also follows  the same way as in Theorem \ref{thm:primesofb(x)} after considering that the radical is the intersection of the minimal primes and minimal primes of $\mb[\bx]$ have trivial kernels.
\end{proof}

\subsection{The prime congruences of $\zma(\bx)$ and $\zma[\bx]$}\label{sec:z[x]}
The description of the primes and the radical of $\zma(\bx)$ and $\zma[\bx]$ can be easily derived from that of $\mb(\bx)$ and $\mb[\bx]$. The key observation is that $\zma \cong \mb(t)/\<(1+t,t)\>$ and consequently $\zma(\bx) = \mb(t,\bx)/\<(1+t,t)\>$ where $\mb(t,\bx)$ is just the semiring of Laurent polynomials over $\mb$ with $k+1$ variables $(t,x_1,\dots,x_k)$. Hence prime congruences of  $\zma(\bx)$ can be identified with the prime congruences of  $\mb(t,\bx)$ containing $(t,1+t)$. By Theorem \ref{thm:primesofb(x)} these are of the form $P(U)$ where $U$ is an admissible matrix with $k+1$ columns, such that the either its first column has all $0$ entries or the first non-zero entry of the first column is positive. We will call such a matrix {\it z-admissible}, and we will denote the congruence defined by it in $\zma(\bx)$ by $P(U)_{\mz}$ and its restriction to $\zma[\bx]$ by $P[U]_{\mz}$. \par\smallskip
By the Newton polytope, $newt(f)$, of a polynomial $f = \sum_i t^{c_i}\bx^{\bn_i}$ in $\zma(\bx)$ or $\zma[\bx]$, we mean the convex hull of the points $[c_i, \bn^i] \in \mz^{k+1}$. We define the {\it hat} of $newt(f)$ to be the set $$\overline{newt(f)} = \{(y_0,\dots,y_k) \in newt(f)\mid \forall z>y_0:\;(z,y_1,\dots,y_k) \notin newt(f) \}. $$
We have the following theorem:

\begin{theorem}\label{thm:primesofz(x)} For the k-variable polynomial semiring $\zma[\bx]$ and the k-variable Laurent polynomial semiring $\zma(\bx)$ we have that:
\begin{itemize}
\item [(i)] The minimal primes of $\zma(\bx)$ (resp. $\zma[\bx]$) are of the form  $P(U)_{\mz}$ (resp. $P[U]_{\mz}$) for a z-admissible matrix $U$ with $k+1$ columns satisfying $Ker(U) \cap \mz^{k+1} = \{\b0\}$.
\item[(ii)] $dim(\zma(\bx)) = dim(\zma[\bx]) = k+1$
\item[(iii)] For any $f,g \in \zma(\bx)$ (resp. $f,g \in \zma[\bx]$) the pair $(f,g)$ lies in the radical of the trivial congruence of  $\zma(\bx)$ (resp. $\zma[\bx]$) if and only if $\overline{newt(f)} = \overline{newt(g)}$.
\item[(iv)] Every radical congruence of $\zma(\bx)$ is QC. $Rad(\diag)$ in $\zma[\bx]$ is QC.
\end{itemize}
\end{theorem}

\begin{proof}
(i) and (ii) follows from the discussion preceding the theorem. For (iii) by the same argument as in the proof of Theorem \ref{thm:primesofb(x)} we need to show that the vertices of $\overline{newt(f)}$ are precisely the exponents of the monomials of $f$ that are maximal with respect to the ordering in the quotient of some minimal prime.  By (i) we have that in both cases minimal primes correspond to term orderings of the variables $(t,\bx)$ such that $1 < t$ and it is clear that the leading monomial of $f$ with respect to such a term ordering has to be one of the vertices lying on $\overline{newt(f)}$. For the other direction for a vertex $v$ on $\overline{newt(f)}$ let $\bu$ be a linear combination with positive coefficients of the outwards pointing normal vectors of the $k$-dimensional faces of $\overline{newt(f)}$ containing $v$, such that the first coordinate of $\bu$ is positive. Such a $\bu$ can be chosen since the outwards pointing normal vector of any $k$-dimensional face of  $\overline{newt(f)}$ have positive first coordinate, so if we set the coefficients corresponding to those faces large enough $\bu$ will also have a positive first coordinate. Moreover, $v$ is the unique vertex that maximizes the scalar product taken with $\bu$ on $\overline{newt(f)}$. Hence we can choose a z-admissible matrix $U$ with $\bu$ as its first row and $Ker(U) \cap \mz^{k+1} = \{\b0\}$ and in the term ordering defined by $U$ the leading term of $f$ will be the monomial with exponent $v$. Finally, (iv) follows the same way as in Theorems \ref{thm:primesofb(x)} and \ref{thm:primesofb[x]}.
\end{proof}
\subsection{The prime congruences of $\rma(\bx)$ and $\rma[\bx]$}\label{sec:T[x]}

In this section we describe the primes and the radical of the semirings of polynomials and Laurent polynomials with coefficients in $\rma$. 

A matrix $U$ whose first column has either all zero entries or its first non-zero entry is positive can define a prime congruence $P(U)_{\rma}$ of $\rma(\bx)$, which, as in the previous cases is generated by pairs $(t^{c_1}\bx^{\bn_1} + t^{c_2}\bx^{\bn_2}, t^{c_2}\bx^{\bn_2})$ such that $U((c_2,\bn_2)-(c_1,\bn_2))$ is either the $\b0$ vector or its first non-zero coordinate is positive. Clearly if $U$ is z-admissible and we consider $\zma(\bx)$ as a subsemiring of $\rma(\bx)$, we have $P(U)_{\rma}|_{\zma(\bx)}=P(U)_{\mz}$. However $P(U)_{\rma}$ might not be the only congruence that restricts to $P(U)_{\mz}$ as shown by the following example:

\begin{example}{\rm
Let $r \in \mr$ be an irrational number and let $U$ be the matrix that consists of the single line $[1\ r]$. Since $Ker(U) \cap \mz^2 = \{\b0\}$, $U$ defines a total ordering on $\mz^2$ and hence $P(U)$ is a minimal prime of $\mb(x_1,x_2)$ and $P(U)_{\mz}$ is a minimal prime of $\zma(x_1)$. Consequently any subsequent rows to $U$ would be redundant. However $Ker(U) \cap \mr\oplus\mz \neq \{0\}$, so $U$ does not define a total ordering on the monomials of $\rma(x_1)$, and one can add a subsequent row to $U$ which will give the ordering on the elements in $Ker(U) \cap (\mr \oplus \mz)$. For example denoting by $U_+$ the matrix which is obtained from $U$ by adding the row $[0\ 1]$ and $U_-$ the matrix which is obtained by adding the row $[0\ -1]$, we have that $P(U_+)_{\rma}$ and $P(U_+)_{\rma}$ are distinct minimal primes of $\rma(x_1)$ both strictly containing $P(U)_{\rma}$, and $P(U_+)_{\rma}|_{\zma(\bx)}=P(U_-)_{\rma}|_{\zma(\bx)}=P(U)_{\rma}|_{\zma(\bx)}=P(U)_{\mz}$}.
\end{example}

Motivated by this example we define an $l\times (k+1)$ matrix $U$ to be {\it t-admissible} if its rows are non-redundant with respect to the ordering defined on $\mr\oplus \mz^k$, i.e.  for every $1\leq i \leq l$ there is a ${\pmb v} \in \mr\oplus \mz^k$ such that the $i$-th is the first non-zero entry of $U{\pmb v}$; moreover, we require that in the first column of $U$ either all of the entries are $0$ or its first non-zero entry is positive. Clearly z-admissible matrices are also t-admissible, but some t-admissible matrices, like $U_+$ and $U_-$ from the above example, might not be z-admissible. Then the prime congruence $P(U)_{\rma}$ is defined for all t-admissible matrices $U$, and $P(U)_{\rma}|_{\zma(\bx)}=P(U')_{\mz}$ where $U'$ is the matrix we obtain from $U$ after removing rows that become redundant when $U$ defines an ordering of the monomials with coefficients in $\zma$.
The restriction of $P(U)_{\rma}$ to $\rma[\bx]$ will be denoted by $P[U]_{\rma}$. As previously, we aim to show that all primes of $\rma(\bx)$ are of the form $P(U)$ for a t-admissible $U$. For this we will need the following variation on the result from \cite{Rob85} which we recalled in Proposition \ref{prop: Robbiano}. 

\begin{lemma}\label{lem:RobbforT}
For any group ordering $\preceq$ on the multiplicative group of the monomials of $\rma(\bx)$ satisfying that for every $c_1,c_2\in \mr$ and $\bn\in \mz^k$ we have that $t^{c_1}x^{\bn_1} \preceq t^{c_2}x^{\bn_2}$ if and only if $c_1 \leq c_2$ by the usual ordering on $\mr$, there exits a t-admissible matrix $U$ such that $t^{c_1}x^{\bn_1} \prec t^{c_2}x^{\bn_2}$ if and only if the first non-zero coordinate of $U((c_2,\bn_2)-(c_1,\bn_1))$ is positive. 
\end{lemma}
\begin{proof}
First note that the multiplicative group of the monomials of $\rma(\bx)$ is isomorphic to the additive group $(\mr \oplus \mz^k, +)$. It follows from Lemma 1 of \cite{Rob85} (and can also be easily checked) that every group ordering of $(\mr \oplus \mz^k, +)$ uniquely extends to a group ordering of $G = (\mr \oplus \mq^k, +)$. By a slight abuse of notation let us denote the ordering induced on $G$ by $\preceq$ as well. Let $G_+$ denote the set $\{\bv\in G|\bv \succ \b0\}$ and $G_-$ denote the set $\{\bv\in G|\bv \prec \b0\}$. Now following the original argument from \cite{Rob85} we define $I_G$ to be the set of points $p \in \mr^{k+1}$ such that each open (Euclidean) neighbourhood of $p$ contains elements from both $G_+$ and $G_-$. It is easy to verify that $I_G$ is a linear subspace. Let $V_+$ (resp. $V_-$) denote the open set in $\mr^{k}$ that consists of points with an open neighbourhood that does not intersect $G_-$ (resp. $G_+$). Now we have that $\mr^{k+1}\setminus I_G = V_- \cup V_+$, so the complement of $I_G$ is the union of disjoint open sets and hence disconnected, it follows that $dim(I_G) \geq k$. On the other hand $V_+$ and $V_-$ each contain at least an open quadrant, so $dim(I_G) = k$. Let us note that this is where the argument would fail if one wanted to extend it to an arbitrary group ordering on $\mr \oplus \mz^k$, but in our case, due to the elements of $\mr \oplus \{\b0\}$ being ordered in the usual way, for the vector ${\pmb e_0} = (1,0,\dots,0)$ and a $\mz$-basis ${\pmb e_1},\dots,{\pmb e_k}$ of $\mz^k$ satisfying ${\pmb e_i} \succ \b0$, we have that the positive $\mr$-linear combinations of ${\pmb e_0},\dots,{\pmb e_k}$ are indeed in $V_+$ and the negatives of these are in $V_-$. Now for the normal vector $\bu$ of $I_G$ pointing towards $V_+$ and any $\bv_1,\bv_2 \in G$ we have that $\bu\cdot(\bv_2-\bv_1) > 0 \Rightarrow \bv_1 \prec \bv_2$, where $\cdot$ denotes the usual scalar product on $\mr^{k+1}$, so $\bu$ can be chosen as the first row of $U$. Moreover, the subgroup $G_0 = \{\bv\in G|\bu\cdot\bv=0 \}$ is isomorphic to $\mz^k$ when the first coordinate of $\bu$ is non-zero, and it is isomorphic to $\mr \oplus \mz^l$ for some $l<k$ if the first coordinate of $\bu$ is zero. Hence either by Proposition \ref{prop: Robbiano} or by induction we have that the ordering on $G_0$ is given by a matrix with at most $k$ rows, and by adding to that matrix $\bu$ as a first row we obtain the $U$ in the lemma.
\end{proof}

In the following proposition we will list the analogues of Propositions \ref{prop: orderingbasic}/(iii), \ref{prop:kernels}, \ref{prop: primes of term orderings}, \ref{prop:dimofP[U]} and Lemma \ref{lem: minprimes} for $\rma(\bx)$ and $\rma[\bx]$. We will omit the proofs since they are essentially the same as in the previous section. 
Recall that $U(i)$ denotes the matrix that consists of the first $i$ rows of $U$.

\begin{proposition}\label{prop:propsofT}
\begin{itemize}
\item [(i)] Primes of $\rma(\bx)$ always have $\{0\}$ as their kernel, and the kernel of a prime in $\rma[\bx]$ is generated by a subset of the variables $\bx$.
\item[(ii)] If $Q$ is a prime congruence of $\rma[\bx]$ with $Ker(Q) = \{0\}$, then $Q = P|_{\rma[\bx]}=P$ for some prime congruence $P$ of of $\rma(\bx)$.
\item[(iii)] Every congruence of $\rma(\bx)$ containing some $P(U)_{\rma}$ for an $l\times (k+1)$ t-admissible matrix $U$ is of the form $P(U(i))_{\rma}$) for some $0\leq i \leq l$. 
\item[(iv)] For an $l\times (k+1)$ t-admissible matrix $U$, we have  $dim(\rma(\bx)/P(U)_{\rma}) = dim(\rma[\bx]/P[U]_{\rma}) = r(U) = l$.
\item[(v)] Every prime of $\rma(\bx)$, contains a prime $P(U)_{\rma}$ for a t-admissible matrix $U$ with $Ker(U) \cap \mr \oplus \mz^{k} = \{\b0\}$.
\end{itemize}
\end{proposition}

Similarly to the previous cases the Newton polytope, $newt(f)$, of a polynomial $f = \sum_i t^{c_i}\bx^{\bn_i}$ in $\rma(\bx)$ or $\rma[\bx]$, we mean the convex hull of the points $[c_i, \bn_i] \in \mr\oplus\mz^{k}$. The hat of the Newton polytope is defined the same way as in the case of $\zma(\bx)$.\par\smallskip
Now we are ready to describe the primes and the radicals of $\rma[\bx]$ and $\rma(\bx)$, which is analogous to the previous cases studied, except that this time we need to consider t-admissible matrices for defining prime congruences. 

\begin{theorem}\label{thm: primesofT[x]} For the k-variable polynomial semiring $\rma[\bx]$ and the k-variable Laurent polynomial semiring $\rma(\bx)$ we have that:,
\begin{itemize}
\item[(i)]  Every prime congruence of $\rma(\bx)$ is of the form $P(U)_{\rma}$ for a t-admissible matrix $U$. For every prime congruence $P$ of  $\rma[\bx]$ there is a (possibly empty) subset $H$ of the variables $\bx$ and a prime $P[U]$ of the polynomial semiring $\rma[\bx']$ with variables $\bx'=\bx\setminus H$, such that  $P$ is generated by the pairs $\{(x_i,0)|\:x_i\in H\}$ and the image of $P[U]$ under the embedding $\rma[\bx']\hookrightarrow \rma[\bx]$.
\item[(ii)] The minimal prime congruences of $\rma[\bx]$ have $\{0\}$ as their kernel. Every minimal prime of $\rma[\bx]$ (resp. $\rma(\bx)$)  is of the form  $P[U]_{\rma}$ (resp. $P(U)_{\rma}$), where $U$ is a t-admissible matrix with $Ker(U)\cap \mr\oplus\mz^k = \{\b0\}$.
\item[(iii)] $dim(\rma(\bx)) = dim(\rma[\bx]) = k+1$. 
\item[(iv)] For any $f,g \in \rma(\bx)$ (resp. $f,g \in \rma[\bx]$) the pair $(f,g)$ lies in the radical of the trivial congruence of  $\rma(\bx)$ (resp. $\rma[\bx]$) if and only if $\overline{newt(f)} = \overline{newt(g)}$.
\item[(v)] Every radical congruence of $\rma(\bx)$ is QC. $Rad(\diag)$ in $\rma[\bx]$ is QC.
\end{itemize}
\end{theorem}
\begin{proof}
(ii) follows from Lemma \ref{lem:RobbforT}, and the rest of the theorem follows from Proposition \ref{prop:propsofT} by the same arguments as in Theorems \ref{thm:primesofb(x)}, \ref{thm:primesofb[x]} and \ref{thm:primesofz(x)}.
\end{proof}

\section{Finitely generated congruences of polynomials and a Nullstellensatz for $\rma[\bx]$}\label{sec:Nullst}

In this section we give an improvement of the result of A. Bertram and R. Easton from \cite{BE13}, which can be regarded as an analogue of the Nullstellensatz. Following their notation, for a congruence $E$, we will denote by $V(E)$ the set of points $\ba \in \rma^k$ for which $f(\ba) = g(\ba)$ for every $(f,g) \in E$. Furthermore, for a subset $H \subseteq \rma^k$ we will denote by $\bE(H)$ the congruence $\{(f,g)\mid f(\ba) = g(\ba),\forall \ba \in H\}$. The focus of \cite{BE13} is to describe the congruence $\bE(V(E))$.\par\smallskip
To put this in our context first note that if $\ba = (t^{d_1},\dots,t^{d_k}) = t^\bd$ is a point in $\rma^k$ such that all of its coordinates are non-zero and $m = t^c\bx^{\bn}$ is a monomial in $\rma[\bx]$, then $m(\ba) = t^{c+\sum_i(d_i n_i)} = t^{(c,\bn) (1,\bd)}$. Hence $\bE(\{\ba\}) = P[U]_{\rma}$ for the matrix $U$ consisting of the single row $(1,d_1,\dots,d_k)$. Similarly, when some of the coordinates of $\ba$ are zero $Ker(\bE(\{\ba\})$ will be generated by the variables corresponding to the zeros of $\ba$, and $\bE(\{\ba\})$ restricted to the rest of the variables will be defined by the matrix whose single row is $(1,d'_1,\dots,d'_i)$, where the $d'_1,\dots,d'_i$ are the exponents of the non-zero entries of $\ba$. We will call the congruences  $\bE(\{\ba\})$ {\it geometric congruences}. Note that these are precisely the congruences whose quotient is $\rma$. With this terminology the congruence $\bE(V(E))$ is just the intersection of all geometric congruences containing $E$.\par\smallskip
\begin{remark}{\rm
As we point out in Remark \ref{rem:zariski-topology} the set of primes can be endowed with the Zariski topology in the usual way. It is not difficult to check that when this topology is restricted to the geometric primes one obtains the usual Euclidean topology on $\rma^k$. }
\end{remark}
 In \cite{BE13} for a congruence $E$ the set $E_+$ was defined to consist of all pairs $(f,g)$ for which there exist $1 \neq \epsilon \in \rma$, $h \in \rma[\bx]$ and a non-negative integer $i$, such that: $$(1,\epsilon)({(f,g)^*}^i + (h,0))(f,g) \in E.$$ It was shown in Theorem 3 in \cite{BE13}, and the discussion preceding it that $E \subseteq E_+ \subseteq \bE(V(E))$ and $V(E) = V(E_+)$, moreover, whenever $E$ is finitely generated the set $V(E)$ is empty if and only if $E_+ = \rma[\bx]\times\rma[\bx]$. However, it was left open whether one has $E_+ = \bE(V(E))$ for all finitely generated $E$ or if the set $E_+$ is a congruence in general. The aim of this section is to show that the answer to both these questions is positive. Furthermore, we will show that in each of the cases we studied, radicals of finitely generated congruences are the intersection of the primes with 1-dimensional quotients. 

We will need the following proposition:

\begin{proposition}\label{prop: E+}
\begin{itemize}
\item [(i)] For a $\mb$-algebra $A$, a pair $\alpha \in A \times A$ and a congruence $E$ with $GP(\alpha) \cap E \neq \emptyset$, there is a non-negative integer $i$ and an element $h\in A$ such that $({\alpha^*}^i + (h,0))\alpha \in E$.
\item[(ii)] For a congruence $E$ of $\rma[\bx]$ and any $\epsilon \in \rma\setminus\{1,0\}$ we have that $$E_+ = \{(f,g)\in \rma[\bx]\times\rma[\bx] | \;\gp((1,\epsilon)(f,g)) \cap E \neq \emptyset \} = \{(f,g)\mid\;(f,g)(1,\epsilon) \in Rad(E)\}. $$
\end{itemize}
\end{proposition}
\begin{proof}
For (i), if $GP(\alpha) \cap E \neq \emptyset$, then by definition we have non-negative integers $i,j$ and a $h \in A$ such that $\beta := ({\alpha^*}^i + (h,0))\alpha^j \in E$. If $j\leq 1$ we are done, let us assume $j>1$. After expanding, we obtain that in the quotient $A/E$ we have $$\alpha_1^{i+j} + h\alpha_1^j \leq \beta_1 = \beta_2 \leq \sum_{s=1}^{s=i+j} \alpha_1^{i+j-s}\alpha_2^s + h\sum_{s=1}^{s=j}\alpha_1^{j-s}\alpha_2^s.$$ 
Now set $h' = h(\alpha_1+\alpha_2)^{j-1}$ and $\gamma := ({\alpha^*}^{i+j-1} + (h',0))\alpha$. After expanding the parenthesis, we obtain:$$\gamma_1 = \sum_{s=1}^{s=i+j} \alpha_1^{s}\alpha_2^{i+j-s} + h\sum_{s=1}^{s=j}\alpha_1^{s}\alpha_2^{j-s}$$  $$\gamma_2 = \sum_{s=1}^{s=i+j} \alpha_1^{i+j-s}\alpha_2^s + h\sum_{s=1}^{s=j}\alpha_1^{j-s}\alpha_2^s$$ 
We see that the only terms appearing in $\gamma_1$ but not in $\gamma_2$ are $\alpha_1^{i+j}$ and $h\alpha_1^j$, so comparing with the previous inequality we obtain that in the quotient $A/E$ we have $\gamma_2 \geq \gamma_1$ and then by a symmetric argument $\gamma_2 = \gamma_1$, hence $\gamma \in E$.\par\smallskip
For (ii) first note that a prime congruence contains the pair $(1,\epsilon)$ for an $\epsilon \in \rma\setminus\{1,0\}$ if and only if its defining matrix has all zero entries in the first column.  Now by Proposition \ref{prop:rad corol1} the set $F:= \{(f,g)\in \rma[\bx]\times\rma[\bx] | \;\gp((1,\epsilon)(f,g)) \cap E \neq \emptyset \}$ is the intersection of the prime congruences containing $E$ but not containing $(1,\epsilon)$ so by the previous comment it does not depend on the choice of $\epsilon$. Furthermore, we have $$(1,\epsilon)({(f,g)^*}^i + (h,0))(f,g) \in \gp((1,\epsilon)(f,g))$$
hence $E_+ \subseteq F$. For the other inclusion if $(f,g) \in F$ then by (i) we have an integer $i$ and a $h \in \rma[\bx]$ such that $$({(1,\epsilon)^*}^i{(f,g)^*}^i + (h,0))(1,\epsilon)(f,g) \in E.$$
Now since $(1+\epsilon)$ has a multiplicative inverse for any $\epsilon \in \rma$, after multiplying the above expression with $1/(1+\epsilon)^i$ we obtain that $(f,g) \in E_+$. The second equality follows from Proposition \ref{prop:rad corol1}.
\end{proof}

We will denote the $i$-th row of the matrix $U$ by $U[i]$. For an $l\times k$ admissible (resp. z-admissible, t-admissible) matrix $U$ and a vector $\bw = (w_1,\dots,w_l) \in \mr_+^l$, $P[\bw U]$ (resp. $P[\bw U]_{\zma}$, $P[\bw U]_\rma$) will denote the prime defined by the matrix consisting of the single row $\bw U = \sum_i w_i U[i]$. Note that since the coefficients $w_i$ are positive and the rows of an admissible matrix are linearly independent $\bw U$ will be also admissible (resp. z-admissible, t-admissible). The following lemma holds by identical arguments over all polynomial and Laurent polynomial semirings we have studied so far, to simplify its formalization we will denote by $P(U)_*$ one of $P(U)$, $P[U]$, $P(U)_{\zma}$, $P[U]_{\zma}$, $P(U)_\rma$ or $P[U]_\rma$ depending on which semiring is being considered.

\begin{lemma}\label{lem: 1dimprimes}
Let $P(U)_*$ be a prime with trivial kernel in one of $\mb(\bx)$, $\mb[\bx]$, $\zma(\bx)$, $\zma[\bx]$, $\rma(\bx)$ or $\rma[\bx]$. Then for any pair $(f,g)$ we have that $(f,g) \in P(U)_*$ if and only if there exist positive real numbers $r_1,\dots,r_{l-1}$ such that for any $\bw \in \mr_+^l$ satisfying $w_i / w_{i+1} > r_i\;(\forall i:\;1 \leq i \leq l-1)$, we have $(f,g) \in P(\bw U)_*$.
\end{lemma}
\begin{proof}
We will prove the proposition for polynomials in $\mb(\bx)$ and note that it holds by identical arguments for all of the semirings listed. Let $f = \sum_i \bx^{\bn_i}$ a polynomial in $\mb(\bx)$, and recall that since the quotient of any prime is totally ordered $f$ will be congruent in any prime to one or more of its monomials. Now it is easy to verify that if we pick $r_i$ large enough then for any $w$ satisfying  $w_i / w_{i+1} > r_i$ for all $1 \leq i \leq l-1$ and any $\bn_i,\bn_j$ appearing as exponents in $f$ we have that $\bw U \bn_i \geq \bw U \bn_j$ if and only if either $U\bn_i = U\bn_j$ or for the smallest $s$ such that $U[s]\bn_i \neq U[s]\bn_j$ we have $U[s]\bn_i > U[s]\bn_j$. It follows that for large enough $r_i$-s and a $\bw$ as in the proposition, the leading terms of both $f$ and $g$ in $P(\bw U)$ are the same as in $P(U)$, hence $(f,g) \in P(U)$ if and only if $(f,g) \in P(\bw U)$.
\end{proof}

\begin{theorem}\label{thm: 1dim}
\begin{itemize}
\item [(i)] For a finitely generated congruence $E$ in one of $\mb(\bx)$, $\mb[\bx]$, $\zma(\bx)$, $\zma[\bx]$, $\rma(\bx)$ or $\rma[\bx]$, we have that $Rad(E)$ is the intersection of the primes that contain $E$ and have a quotient of dimension at most $1$.
\item [(ii)] In $\rma[\bx]$, for any finitely generated congruence $E$, we have $E_+ = \bE(V(E))$.
\end{itemize}
\end{theorem}
\begin{proof}
For (i) let $E$ be a congruence generated by the pairs $\{(f_1,g_1),\dots,(f_s,g_s)\}$. By definition we have that $Rad(E) = \cap\{P \mid \mbox{P prime,}\;(f_i,g_i)\in P\;\forall i\}$. If $P(U)_*$ is a prime with trivial kernel and a quotient of dimension $l \geq 2$, containing all of the $(f_i,g_i)$ then we can choose $(r_1,\dots,r_{l-1})$ that are large enough for all of the $(f_i,g_i)$ in the setting of Proposition \ref{lem: 1dimprimes}. Denoting by $W$ the set of vectors $\bw \in \mr_+^l$ satisfying $w_i / w_{i+1} > r_i$ for all $1 \leq i \leq l-1$, it follows that $(f_i,g_i) \in P(\bw U)_*$ for all $1\leq i \leq s$ and $\bw \in W$. Moreover, by applying the other direction of Proposition \ref{lem: 1dimprimes} we also have that $\cap_{\bw \in W}  P(\bw U)_* \subseteq P(U)_*$, hence $P(U)_*$ can be removed from the intersection defining $Rad(E)$. We can argue the same way in the case when $P(U)_*$ has non-trivial kernel by considering it in the polynomial subsemiring generated by the variables that are not in $Ker(P(U)_*)$.\par\smallskip
For (ii) by Proposition \ref{prop: E+} and Proposition \ref{prop:rad corol1} we have that $E_+$ is the intersection of the primes that contain $E$ but not contain $(1, \epsilon)$ for any $\epsilon \in \rma\setminus \{1\}$, and by the discussion at the start of this section it follows that $\bE(V(E))$ is the intersection of the geometric congruences containing $E$, which are exactly those primes that have quotients with dimension $1$ and not contain the pair $(1, \epsilon)$ for any $\epsilon \in \rma\setminus \{1\}$. Note that $(1,\epsilon)$ for $\epsilon \in \rma\setminus \{1\}$ is contained in a prime precisely when its defining matrix has all zeros in the first column, thus if $(1,\epsilon) \notin P[U]_\rma$ then $(1,\epsilon) \notin P[\bw U]_\rma$ for any vector $\bw$ with positive entries. Now one can argue the same way as for (i).
\end{proof}

We conclude this section with a statement showing that the polynomials that agree on every point of $\rma^k$ are precisely the pairs that are in $Rad(\diag)$. This is essentially the same as Theorem 1 of \cite{BE13}, but our proof is different. 

\begin{proposition}
$\bE(\rma^k) = \diag_+ = Rad(\diag)$.
\end{proposition}
\begin{proof}
The first equality follows from Theorem \ref{thm: 1dim}. For the second equality since $\diag_+$ is the intersection of a subset of all primes we clearly have $Rad(\diag) \subseteq \diag_+$. For the other inclusion, if $(f,g) \notin Rad(\diag)$ then by Theorem \ref{thm: primesofT[x]} we have that for one of them, say $f$, there is a vertex $v$ on $\overline{newt(f)}$ that lies outside of $\overline{newt(g)}$. Now by the same argument as in the proof of Theorem \ref{thm:primesofz(x)} one can pick a vector $\bu$ with positive first entry such that $v$ is the unique vertex that maximizes the scalar product taken with $\bu$ on $\overline{newt(f)}$. Now let $U$ be a t-admissible matrix with $\bu$ as its first row such that $P[U]_\rma$ is a minimal prime. Since in $P(U)_\rma$ each equivalence class contains precisely one monomial and $f$ is congruent to the monomial with exponent $v$ we have  $(f,g) \notin P[U]_\rma$. Moreover, since the first entry of $\bu$ is nonzero $(1,\epsilon) \notin P[U]_\rma$ for any $\epsilon \in \rma\setminus \{1\}$. Now since by Proposition \ref{prop: E+} and Proposition \ref{prop:rad corol1} $\diag_+$ is the intersection of all primes that do not contain $(1,\epsilon)$ for $\epsilon \in \rma\setminus \{1\}$, we have that $\diag_+ \subseteq P[U]_\rma$ and consequently $(f,g) \notin \diag_+$.
\end{proof}


\bibliographystyle{amsplain}

\begin{thebibliography}{mmm}

\bibitem[BE13]{BE13}
    A. Bertram and R. Easton,
    \emph{The Tropical Nullstellensatz for Congruences}, Advances in Mathematics 308 (2017) 36-82
    
\bibitem[CC13]{CC13}
	A. Connes and C. Consani,
	\emph{Projective geometry in characteristic one and the epicyclic category}, Nagoya Mathematical Journal 217 (2015), 95-132.
	
\bibitem[CC14]{CC14}	
	A. Connes and C. Consani,
    \emph{The Arithmetic Site}, Comptes Rendus Mathematique Ser. I 352
(2014), 971-975.

\bibitem[IR14]{IR14}
    Z. Izhakian and L. Rowen,
    \emph{Congruences and coordinate semirings of tropical varieties},  Bulletin des Sciences Math\'ematiques Volume 140, Issue 3 (2016), 231-259

    


\bibitem[GG13]{GG13}
	J. Giansiracusa and N. Giansiracusa,
	\emph{Equations of tropical varieties}, Duke Math. J. 165, no. 18 (2016)



	
\bibitem[Les12]{Les12}
	P. Lescot,
    \emph{Absolute Algebra III-The saturated spectrum}, Journal of Pure and Applied Algebra 216 (2012), no. 7, 1004-1015. 

\bibitem[Lor12]{Lor12}
    O. Lorscheid,
    \emph{The geometry of blueprints: Part I: Algebraic background and scheme theory}, Advances in Mathematics 229 (2012), no. 3, 1804-1846.

\bibitem[MR14]{MR14}
    D. Maclagan and F. Rinc\'{o}n,
	\emph{Tropical schemes, tropical cycles, and valuated matroids}, arXiv:1401.4654 
	

\bibitem[MS]{MS}
    D. Maclagan and B. Sturmfels,
    \emph{Introduction to Tropical Geometry}, Graduate Studies in Mathematics, American Mathematical Society, Providence, RI, vol. 161, 2015

\bibitem[Mik06]{Mik06}
    G. Mikhalkin, 
    \emph{Tropical geometry and its applications}, International Congress of Mathematicians. Vol. II, Eur. Math. Soc., Z{\"u}rich, 2006, 827-852. MR 2275625 (2008c:14077)

	
\bibitem[Rob85]{Rob85}
	L. Robbiano, 
	\emph {Term orderings on the polynomial ring}, EUROCAL ’85, Vol. 2 (Linz, 1985), Lecture Notes in Comput. Sci. 204, 513-517, Springer, Berlin (1985)


\bibitem[Tit56]{Tit56}
   J. Tits,
   \emph{Sur les analogues alg\'ebriques des groupes semi-simples complexes},
   Colloque d'alg\`ebre sup\'erieure, tenu \`a Bruxelles du 19 au 22 d\'ecembre 1956, Centre Belge de Recherches Math\'ematiques \'Etablissements Ceuterick, Louvain; Librairie Gauthier-Villars, Paris (1957), 261-289.
\end{thebibliography}
\vfill\eject

\end{document}